\documentclass[journal, 10pt]{IEEEtran}
\date{}
\usepackage{amsmath}   
\usepackage{amsthm}
\usepackage{tabularx,amsmath,amssymb,graphicx,paralist,subfigure,pdfpages,cite,algorithm,algpseudocode,paralist,multirow,comment}
\usepackage[autostyle]{csquotes}
\newtheorem{lemma}{Lemma} 
\newtheorem{theorem}{Theorem}

\newtheorem{prop}{Proposition}
\newtheorem{corollary}{Corollary}
\newtheorem{rmk}{Remark}
\newtheorem{assumption}{Assumption}
\usepackage{lipsum}
\usepackage{verbatim}
\usepackage{algorithm}
\usepackage{algpseudocode}

\usepackage[colorlinks]{hyperref}
\hypersetup{
citecolor = teal
}

\def\R{\mathbb{R}}
\def\mc{\mathcal}
\def\mb{\mathbf}
\def\mbb{\mathbb}
\def\ra{\rightarrow}
\def\E{\mbb{E}}
\def\P{\mathbb{P}}

\def\W{\mathbf{W}}
\def\G{\mb{G}}
\def\J{\mb{J}}
\def\H{\mathbf{H}}
\def\I{\mb{I}}

\def\1{\mathbbm{1}}
\def\F{\mc{F}}

\def\mbb{\mathbb}
\def\mb{\mathbf}
\def\mc{\mathcal}
\def\wh{\widehat}
\def\ul{\widetilde}
\def\ol{\overline}
\def\ul{\underline}

\def\bds{\boldsymbol}
\newcommand{\mn}[1]{{\left\vert\kern-0.25ex\left\vert\kern-0.25ex\left\vert\kern0.3ex #1 
    \kern0.3ex\right\vert\kern-0.25ex\right\vert\kern-0.25ex\right\vert}}
    
\usepackage{bbm}
\renewcommand{\arraystretch}{1.4}

\def\mbb{\mathbb}
\def\mb{\mathbf}
\def\mc{\mathcal}
\def\wh{\widehat}
\def\ol{\overline}
\def\ul{\underline}

\def\x{\mb{x}}
\def\y{\mb{y}}

\def\a{\alpha}
\def\g{\mb{g}}
\def\nf{\nabla\mb{f}}
\def\ra{\rightarrow}
\def\bds{\boldsymbol}

\newenvironment{P1}
  {\begin{proof}[Proof of Theorem~\ref{main_ncvx}]}
  {\end{proof}}
 
\newenvironment{P2}
  {\begin{proof}[Proof of Theorem~\ref{main_PL}]}
  {\end{proof}}

\begin{document}

    \title{A fast randomized incremental gradient method for decentralized non-convex optimization}
\author{Ran Xin,  Usman A. Khan, and Soummya Kar
		\thanks{RX and SK are with the Electrical and Computer Engineering (ECE) Dept. at Carnegie Mellon University, \texttt{\{ranx,soummyak\}@andrew.cmu.edu}. UAK is with the ECE Dept. at Tufts University, \texttt{khan@ece.tufts.edu}. The work of SK and RX has been partially supported by NSF under award \#1513936. The work of UAK has been partially supported by NSF under awards \#1903972 and \#1935555. 
		}
	}

\maketitle
\begin{abstract}
We study decentralized non-convex finite-sum minimization problems described over a network of nodes, where each node possesses a local batch of data samples. {\color{black}In this context, we analyze a single-timescale randomized incremental gradient method, called~\textbf{\texttt{GT-SAGA}}.} \textbf{\texttt{GT-SAGA}} is computationally efficient as it evaluates one component gradient per node per iteration and achieves provably fast and robust performance by leveraging node-level variance reduction and network-level gradient tracking. For general smooth non-convex problems, we show the almost sure and mean-squared convergence of \textbf{\texttt{GT-SAGA}} to a first-order stationary point and further describe regimes of practical significance where it outperforms the existing approaches and achieves a network topology-independent iteration complexity respectively. When the global function satisfies the Polyak-\L ojaciewisz condition, we show that~\textbf{\texttt{GT-SAGA}} exhibits linear convergence to an optimal solution in expectation and describe regimes of practical interest where the performance is network topology-independent and improves upon the existing methods. {\color{black}Numerical experiments are included to highlight the main convergence aspects of \textbf{\texttt{GT-SAGA}} in non-convex settings.} 
\end{abstract}


\section{Introduction}
In this paper, we consider decentralized optimization problems that arise in many  control and modern learning applications where very large-scale and geographically distributed nature of data precludes centralized storage and processing. The problem setup is based on~$n$ nodes communicating over a network modeled as a directed graph~${\mc G := \{\mc{V},\mc{E}\}}$, where ${\mc{V} := \{1,\cdots,n\}}$ is the set of node indices and~$\mc{E}$ is the collection of ordered pairs~${(i,r),i,r\in\mc{V}}$, such that node~$r$ sends information to node~$i$. Each node~$i$ has access to a local, possibly private, collection of~$m$ smooth component functions~${\{f_{i,j}:\mbb{R}^p\rightarrow\mbb{R}\}_{j=1}^m}$ that are non-convex. Each~$f_{i,j}$ can be viewed as a cost incurred by the~$j$-th data sample at the~$i$-th node. The goal of the networked nodes is to agree on a first-order stationary point of the average of all component functions via local computation and communication at each node, i.e., 
\begin{equation}\label{Problem}
\min_{\mb{x}\in\mbb{R}^p}F(\mb x) := \frac{1}{n}\sum_{i=1}^n f_i(\mb x), ~~ f_i(\mb x) := \frac{1}{m}\sum_{j=1}^{m}f_{i,j}(\mb x).
\end{equation}
Decentralized optimization dates back to the seminal contribution of Tsitsiklis in 1980s~\cite{DGD_tsitsiklis}, where the primary focus was on control and signal estimation problems, and has received a resurgence of interest recently due to its promise in large-scale control and machine learning applications~\cite{DSGD_NIPS,SGP_ICML}.

\subsection{Related work}
First-order methods~\cite{book_beck,PIEEE_Xin} that rely mainly on gradient information are commonly used to approach large-scale optimization formulations like Problem~\eqref{Problem}. The works on decentralized first-order methods include several well-known papers on decentralized gradient descent (DGD)~\cite{DGD_nedich,DGD_Kar,diffusion_Chen,DGD_Yin}. Although DGD-type methods are effective in homogeneous environments like data centers, its performance degrades significantly when data distributions across the nodes become heterogeneous~\cite{SPM_Xin}. Decentralized first-order methods that improve the performance of DGD include, e.g., EXTRA~\cite{EXTRA}, Exact Diffusion/NIDS~\cite{SED,NIDS,D2}, DLM~\cite{DLM}, and methods based on gradient tracking~\cite{GT_CDC,NEXT_scutari,harnessing,DIGing,MP_Pu,improved_DSGT_Xin,dual_GT,GT_Wai}; see also general primal-dual frameworks~\cite{GT_jakovetic,PD_AL,PD_Xu,PD_Lu} that unify the aforementioned methods under certain conditions. Some structured formulations have also been considered recently, such as weak convexity~\cite{wcvx_SS}, coupled constraints~\cite{constraint_diffusion}, and coordinate updates~\cite{block_GT}. 

In decentralized batch gradient methods such as~\cite{DGD_Yin,harnessing,EXTRA}, each node~$i$ computes a full batch gradient $\sum_j\nabla f_{i,j}$ at each iteration. Clearly, batch gradient computation becomes expensive when the local batch size~$m$ is large and nodes have limited computational capabilities. Efficient stochastic  methods, e.g.,~\cite{D2,SED,DSGD_Pu,DSGD_vlaski_1,improved_DSGT_Xin,DSGD_Swenson}, thus use randomly sampled component gradients from each local batch; however, the convergence of these methods is typically slower compared with their batch gradient counterparts due to the persistent noise incurred by the stochastic gradients. Towards fast convergence with stochastic gradients, popular variance reduction techniques, e.g.,~\cite{SAGA,SVRG,SARAH,spider,S2GD}, have been adapted to the decentralized settings. {\color{black}For instance, algorithms in~\cite{DSA,DAVRG,GTVR,Network-DANE,AccDVR,SPM_Xin}, including \textbf{\texttt{GT-SAGA}}~\cite{GTVR,SPM_Xin} considered in this paper, are shown to achieve linear rate to the optimal solution for strongly-convex problems; however, the applicability of these methods to non-convex problems remains open.} Recent works~\cite{D_Get,GT-SARAH} propose decentralized variance-reduced methods for non-convex problems. These two methods however require periodic batch gradient evaluations across the nodes in addition to component gradient computations at each iteration; this \textit{two-timescale hybrid scheme} imposes practical implementation challenges, such as periodic network synchronizations, especially over large-scale ad hoc networks.

\subsection{Our contributions}
{\color{black}In this paper, we analyze \textbf{\texttt{GT-SAGA}}, a \textit{single-timescale randomized incremental} gradient method, originally proposed in~\cite{GTVR} for strongly-convex problems, and show that it achieves fast convergence in non-convex settings. At the node level, \textbf{\texttt{GT-SAGA}} adopts a local SAGA-type \cite{SAGA,SAGA_reddi,IAG_mert,IAG_wai,IAG_mert2} randomized incremental approach to obtain variance-reduced estimates of local batch gradients, by leveraging historical component gradient information. At the network level, \textbf{\texttt{GT-SAGA}} employs a gradient tracking mechanism~\cite{GT_CDC,NEXT_scutari} to fuse the local batch gradient estimates, obtained from the local SAGA procedures, to track the global batch gradient. These are the two building blocks that amount to the fast convergence and robustness to heterogeneous data in \textbf{\texttt{GT-SAGA}} for non-convex problems.}
Compared with the existing two-timescale variance reduced methods~\cite{D_Get,GT-SARAH} for decentralized non-convex optimization, \textbf{\texttt{GT-SAGA}} is single-timescale and eliminates completely the need of batch gradient computations and periodic network synchronizations, and is hence much easier to implement especially in ad hoc settings; see Remarks~\ref{rmk_prac} and~\ref{rmk_stor} for further discussion. The main technical contributions in this paper are summarized as follows:

\subsubsection{General smooth non-convex problems} 
For this problem class, we show the asymptotic convergence of \textbf{\texttt{GT-SAGA}} to a first-order stationary point in the almost sure and mean-squared sense. In a big-data regime, where the local batch size~$m$ is very large, \textbf{\texttt{GT-SAGA}} achieves a network topology-independent convergence rate, leading to a non-asymptotic linear speedup compared with the centralized SAGA~\cite{SAGA_reddi} at a single node. In large-scale network regimes, i.e., when the number of the nodes and the network spectral gap inverse are relatively large compared to the local batch size~$m$, we show that \textbf{\texttt{GT-SAGA}} outperforms the existing best known convergence rate~\cite{GT-SARAH}. {\color{black}We also introduce a measure of function heterogeneity across the nodes. Based on this measure, we show that the effect of function heterogeneity on the convergence rate of \textbf{\texttt{GT-SAGA}} appears in a fashion that is separable from the effects of local batch size and the network spectral gap. As a consequence, the effect of function heterogeneity often diminishes when the local batch size is large and/or the connectivity of the network is weak, demonstrating the robustness of \textbf{\texttt{GT-SAGA}} to function heterogeneity. In contrast, the state-of-the-art decentralized non-convex variance-reduced method~\cite{GT-SARAH} does not achieve such separation and hence has worse convergence rate than \textbf{\texttt{GT-SAGA}} when the function heterogeneity is large and the network is weakly connected. These improvements are achieved by leveraging the conditional unbiasedness of SAGA estimators to obtain tighter bounds in the stochastic gradient tracking analysis.
See Remarks~\ref{rmk_hete},~\ref{rmk_bdata}, and~\ref{rmk_network} for details.} 

\subsubsection{Smooth non-convex problems under the global Polyak-\L ojasiewicz (PL) condition} 
For this problem class, we show that~\textbf{\texttt{GT-SAGA}} achieves linear convergence to an optimal solution in expectation. To the best of our knowledge, this is the first linear rate result for decentralized variance-reduced methods under the PL condition, while the existing ones require strong convexity \cite{DSA,DAVRG,Network-DANE,AccDVR,GTVR}. {\color{black}This generalization is non-trivial since the existing analysis essentially uses the unique optimal solution under strong convexity as a reference point to bound related error terms, while the PL condition allows for the existence of multiple optimal solutions.}
In comparison with the existing linearly-convergent, decentralized deterministic batch gradient methods under the PL condition~\cite{GT_zero,PL_DPD,improved_DSGT_Xin}, \textbf{\texttt{GT-SAGA}} provably achieves faster linear rate, in terms of the component gradient computation complexity at each node, when the local batch size~$m$ is large, demonstrating the advantage of the employed variance reduction technique. In a big-data regime where~$m$ is large enough, we show that the linear rate of \textbf{\texttt{GT-SAGA}} becomes network topology-independent. See Remarks~\ref{rmk_pl_bdata} and~\ref{rmk_pl_full} for details.

\subsubsection{Convergence analysis} 
We note that our analysis of SAGA-type variance reduction procedures is different from the existing ones~\cite{SAGA_reddi,prox_SAGA_reddi}, which require careful constructions of Lyapunov functions. We avoid such delicate constructions by adopting a direct analysis approach, based on linear time-invariant (LTI) dynamics, which may be of independent interest and perhaps more readily extendable to other non-convex problems. {\color{black}We note that the LTI dynamics-based analysis has mainly been used in convex problems in the existing literature of gradient tracking methods, e.g.,~\cite{harnessing,MP_Pu}.} 
Somewhat surprisingly, a special case of our analysis, i.e., when the network is complete, provides the first linear rate result of the \textit{original} centralized SAGA algorithm~\cite{SAGA} under the PL condition. Indeed, the existing analysis~\cite{SAGA_reddi,prox_SAGA_reddi} is only applicable to a \textit{modified} SAGA, which periodically restarts and samples its iterates; see Remark~\ref{rmk_pl_c} for details. Finally, our analysis is also substantially different from that of the existing decentralized non-convex variance-reduced methods~\cite{D_Get,GT-SARAH}, where the variances of the stochastic gradients are bounded recursively, due to their hybrid nature. In contrast, we introduce a proper auxiliary sequence to bound the variance of \textbf{\texttt{GT-SAGA}}; see Subsection~\ref{sec_bvr},~\ref{sec_t} for~details.

\subsection{Outline of the paper and notation}
\subsubsection{Outline of the paper} Section~\ref{main} presents the \textbf{\texttt{GT-SAGA}} algorithm and the main convergence results.  Section~\ref{exp} illustrates the main theoretical results with the help of numerical simulations. Section~\ref{conv} presents the convergence analysis.
Section~\ref{clu} concludes the paper. 

\subsubsection{Notation} The set of positive real numbers is by~$\mbb{R}^+$. We use lowercase bold letters to denote vectors and uppercase bold letters to denote matrices. The matrix,~$\mb{I}_d$ (resp.~$\mb{O}_d$), represents the~$d\times d$ identity (resp. zero matrix). The vector, $\mb{1}_d$ (resp.~$\mb{0}_d$), is the~$d$-dimensional ones (resp. zeros). 
The Kronecker product of two matrices is denoted by~$\otimes$. We use~$\|\cdot\|$ to denote the $2$-norm of a vector or a matrix. For a matrix~$\mb{X}$, we use~$\rho(\mb{X})$ to denote its spectral radius
and~$\mbox{diag}(\mb{X})$ as the diagonal matrix with the diagonal entries of~$\mb{X}$. Matrix/vector inequalities are stated in the entry-wise sense. We fix a proper probability space~$(\Omega,\F,\P)$ for all random variables in question and~$\mbb{E}[\cdot]$ denotes the expectation; for an event~$A\in\mc{F}$, its indicator is denoted as~$\1_{A}$. We use~$\sigma(\cdot)$ to denote the~$\sigma$-algebra generated by the random variables and/or events. For two quantities~$A,B\in\mathbb{R}^{+}$, we denote~$A \lesssim B$ if there exists a universal~$c$ such that~$A \leq c B$. 

\section{The \textbf{\texttt{GT-SAGA}} algorithm and main results}\label{main}
{\color{black}\textbf{\texttt{GT-SAGA}}~\cite{GTVR}, built upon local SAGA estimators~\cite{SAGA} and global gradient tracking~\cite{NEXT_scutari,GT_CDC}, is formally presented in Algorithm~\ref{GT-SAGA}. We refer the readers to~\cite{GTVR,SPM_Xin} for detailed discussion on the development of \textbf{\texttt{GT-SAGA}}. In this paper, we require for conciseness that all nodes start at the same point, i.e.,~${\x_i^0 = \ol{\x}^0}, {\forall i\in\mc{V}}$. We emphasize that the complexity results of \textbf{\texttt{GT-SAGA}} established in this paper hold, up to factors of universal constants, for the case where the nodes are initialized differently.}  We comment on the practical implementation aspects of \textbf{\texttt{GT-SAGA}} in comparison with the existing approaches in the following remarks.

\begin{rmk}[\textbf{Single-timescale implementation}]\label{rmk_prac}
\normalfont The existing decentralized variance-reduced methods for non-convex optimization~\cite{GT-SARAH,D_Get} are based on a two-timescale, double-loop implementation. Specifically, these methods, within each inner-loop, run a fixed number of stochastic gradient type iterations, while, at each outer-loop iteration, a local batch gradient is computed at each node. This double-loop nature imposes challenges on the practical implementation of the two methods in~\cite{GT-SARAH,D_Get}. First, periodic batch gradient computation incurs a synchronization overhead on the communication network and jeopardizes the actual wall-clock time when the networked nodes have largely heterogeneous computational capabilities. Second, these two methods have an additional parameter to tune, i.e., the length of each inner loop, other than the step-size. Although this parameter maybe be chosen as~$m$~\cite{GT-SARAH}, 
this particular choice may not lead to the best performance in practice.
In sharp contrast, \textbf{\texttt{GT-SAGA}} admits a simple single-timescale implementation since it only evaluates \textit{one} randomly selected component gradient at each iteration. Furthermore, it only has \emph{one} parameter to tune, i.e., the step-size~$\a$. Therefore, \textbf{\texttt{GT-SAGA}} leads to significantly simpler implementation and tuning compared with the existing decentralized non-convex variance-reduced methods~\cite{GT-SARAH,D_Get}, especially over large-scale ad-hoc networks. {\color{black}Finally, we note that \textbf{\texttt{GT-SAGA}} takes two successive communication rounds per iteration to transmit the state and gradient tracker respectively, as in other gradient tracking-based methods, e.g.,~\cite{improved_DSGT_Xin,MP_Pu,GT-SARAH,D_Get}.}
\end{rmk}

\begin{rmk}[\textbf{Storage requirement}]\label{rmk_stor}
\normalfont
{\color{black}To practically implement \textbf{\texttt{GT-SAGA}}, each node~$i$ needs to retain a gradient table $\{\nabla f_{i,j}(\mb{z}_{i,j}^k)\}_{j=1}^m$ of size~${m\times p}$ in general, which may be expensive. However, for certain structured problems, the size of the gradient table can be largely reduced~\cite{SAGA}. For instance, in non-convex generalized linear models~\cite{nonconvex_BC}, each component function takes the form~$f_{i,j}(\x) = \ell(\mb{x}^\top\bds{\theta}_{i,j})$, where~$\ell:\mathbb{R}\rightarrow\mathbb{R}$ is a non-convex loss and $\bds{\theta}_{i,j}$ is the~$j$-th data at the~$i$-th node. Clearly, $\nabla f_{i,j}(\x) = \ell'(\mb{x}^\top\bds{\theta}_{i,j})\bds{\theta}_{i,j}$ and thus each node~$i$ only needs to retain~$\{\ell'(\mb{z}_{i,j}^\top\bds{\theta}_{i,j})\}_{j=1}^m$, a gradient table of size~$m\times 1$, since the data samples~$\{\bds{\theta}_{i,j}\}_{j=1}^m$ are already stored locally. See Section~\ref{exp_BC} for numerical experiments based on one such example.}
\end{rmk}

We now enlist the assumptions of interest in this paper.

\begin{assumption}\label{sample}
The family~$\{\tau_i^k,s_i^k:i\in\mc{V},k\geq0\}$ of random variables in Algorithm~\ref{GT-SAGA} is independent.
\end{assumption}
Assumption~\ref{sample} is standard in stochastic gradient methods.
{\color{black}Specifically, the index~$s_i^k$ used for updating the gradient table $\{\nabla f_{i,j}(\mb{z}_{i,j}^k)\}_{j=1}^m$ is sampled independently from the index~$\tau_i^k$ used for updating the local SAGA estimator~$\mb{g}_i^k$ per node per iteration. This independence requirement is straightforward to implement and is often posed to simplify the analysis of SAGA type estimators for non-convex problems~\cite{SAGA_reddi,prox_SAGA_reddi}; see Section~\ref{sec_t} for analysis based on this assumption.}

\begin{algorithm}[h]\label{gtsaga}
\caption{\textbf{\texttt{GT-SAGA}} at each node~$i$}
\label{GT-SAGA}
\begin{algorithmic}
\Require{$\mb{x}_i^0 = \ol{\mb{x}}^0\in\mbb{R}^{p}$;~$\alpha\in\mbb{R}^+$; $\{\ul{w}_{ir}\}_{r=1}^n$;~$\mb{z}_{i,j}^0 = \mb{x}_i^0,\forall j\in\{1,\cdots,m\}$;~$\mb{y}_i^0 = \mb{0}_p$;~$\mb{g}_i^{-1} = \mb{0}_p$.}
\For{{$k= 0,1,2,\cdots$}}
\State{Select~$\tau_{i}^{k}$ uniformly at random from~$\{1,\cdots,m\}$;}
\State{Update the local stochastic gradient estimator:\begin{align*}\mb{g}_{i}^{k} = \nabla f_{i,\tau_i^{k}}\big(\mb{x}_{i}^{k}\big) - \nabla f_{i,\tau_i^{k}}\big(\mb{z}_{i,\tau_i^{k}}^{k}\big) +\frac{1}{m}\sum_{j=1}^{m}\nabla f_{i,j}\big(\mb{z}_{i,j}^{k}\big);\end{align*}}
\State{Update the local gradient tracker:$$\mb{y}_{i}^{k+1} = \sum_{r=1}^{n}\ul{w}_{ir}\left(\mb{y}_{r}^{k} + \mb{g}_r^{k} - \mb{g}_r^{k-1}\right);$$}
\State{Update the local estimate of the solution:$$\mb{x}_{i}^{k+1} = \sum_{r=1}^{n}\ul{w}_{ir}\left(\mb{x}_{r}^{k} - \alpha\mb{y}_{r}^{k+1}\right);$$}
\State{Select~$s_{i}^{k}$ uniformly at random from~$\{1,\cdots,m\}$;}
\vspace{0.05cm}
\State{Set~$\mb{z}_{i,j}^{k+1} = \mb{x}_{i}^{k}$ for~$j = s_i^{k}$;~$\mb{z}_{i,j}^{k+1} = \mb{z}_{i,j}^{k}$ for~$j \neq s_i^{k}$;}
\vspace{0.05cm}
\EndFor
\end{algorithmic}
\end{algorithm}


\begin{assumption}\label{smooth}
Each component function~${f_{i,j}:\mbb{R}^p\ra\mbb{R}}$ is differentiable and~$L$-smooth, i.e., there exists~$L>0$, such that
$\left\|\nabla f_{i,j}(\mb{x})-\nabla f_{i,j}(\mb{y})\right\|\leq L\left\|\mb x - \mb y\right\|,$ $\forall\mb{x},\mb{y}\in\mbb{R}^p$,
$\forall i\in\mc{V}$, $\forall j\in\{1,\cdots,m\}.$
Moreover, the global function $F$ is bounded below, i.e., $F^* := \inf_{\mb{x}\in\mbb{R}^p}F(\mb{x}) > -\infty$. 
\end{assumption}
Under Assumption~\ref{smooth}, the local batch functions~$\{f_i\}_{i=1}^n$ and the global function~$F$ are~$L$-smooth. {\color{black}We note that~$L$ stated in Assumption~\ref{smooth} is essentially the maximum of the smoothness parameters of all component functions.} We further consider the case when \emph{the global~$F$} additionally satisfies the Polyak-\L ojasiewicz (PL) condition described below.
\begin{assumption}\label{PL}
The global function~$F:\R^p\ra\R$ satisfies
$2\mu(F(\x) - F^*) \leq \|\nabla F(\x) \|^2$, $\forall \mb{x}\in\mbb{R}^p$, for some~${\mu>0}$.    
\end{assumption}
{\color{black}The PL condition, originally introduced in~\cite{polyak1987introduction}, generalizes the notion of strong convexity to non-convex functions; see~\cite{PL_1} for more discussion.} 
When Assumption~\ref{PL} holds, we denote~$\kappa: = \frac{L}{\mu}\geq1$, which may be interpreted as the condition number of~$F$.
Note that the PL condition implies that every stationary point~$\x^*$ of~$F$, such that~$\nabla F(\mb x^*)=\mb 0_p$, is a global minimizer of~$F$, while~$F$ is not necessarily convex. 

\begin{assumption}\label{network}
The weight matrix $\ul{\W}=\{\ul{w}_{ir}\}\in\mathbb{R}^{n \times n}$ of the network is primitive and doubly-stochastic, i.e., ${\ul{\W}\mb{1}_n = \mb{1}_n}$, ${\mb{1}_n^\top \ul{\W} = \mb{1}_n^\top},$ and $\lambda := \lambda_2(\ul{\W})\in [0,1)$, where~$\lambda_2(\W)$ is the second largest singular value of $\ul{\W}$.
\end{assumption}
Weight matrices that satisfy Assumption~\ref{network} may be designed for strongly-connected, weight-balanced, directed networks or for connected, undirected networks. 
We next discuss the performance metrics of \textbf{\texttt{GT-SAGA}} for different problem classes. For general smooth non-convex problems, we define the iteration complexity of \textbf{\texttt{GT-SAGA}} as the minimum number of iterations required to achieve an~$\epsilon$-accurate stationary point of the global function~$F$, i.e.,
\begin{align*}
\inf\left\{K: \frac{1}{n}\sum_{i=1}^n\frac{1}{K}\sum_{k=0}^{K-1}\E\big[\|\nabla F(\x_i^k)\|^2\big] \leq \epsilon\right\}.
\end{align*}
When the global function~$F$ further satisfies the PL condition, we define the iteration complexity of~\textbf{\texttt{GT-SAGA}} as
\begin{align*}
\inf\left\{k : \E\left[\frac{1}{n}\sum_{i=1}^n\left(F(\x_i^k) - F^*\right)\right] \leq \epsilon\right\}.
\end{align*}
These are standard metrics for decentralized stochastic non-convex optimization methods~\cite{DSGD_NIPS,improved_DSGT_Xin,GT-SARAH,D_Get}. We refer the iteration complexity as the the convergence rate metric of \textbf{\texttt{GT-SAGA}}, since it is the same as the communication and component gradient computation complexity at each node. We are now ready to state the main results of \textbf{\texttt{GT-SAGA}} in the next subsections and discuss their implications.

\subsection{General smooth non-convex functions}
In this subsection, we present the main convergence results of~\textbf{\texttt{GT-SAGA}} for general smooth non-convex functions. 

\begin{theorem}\label{main_ncvx}
Let Assumptions~\ref{sample},~\ref{smooth}, and~\ref{network} hold. If the step-size~$\a$ of \textbf{\texttt{GT-SAGA}} satisfies~$0<\a\leq\ol{\a}_1$, where $$\ol{\a}_1:=\min\left\{\frac{(1-\lambda^2)^2}{48\lambda},\frac{2n^{1/3}}{13m^{2/3}},\frac{1}{2},\frac{(1-\lambda^2)^{3/4}}{18\lambda^{1/2}m^{1/2}}\right\}\frac{1}{L},$$ then all nodes asymptotically agree on a stationary point
in both mean-squared and almost sure sense, i.e.,~$\forall i,r\in\mc{V}$,
\begin{align*}
&\P\Big(\lim_{k\rightarrow\infty}\|\x_i^k - \x_r^k\| = 0\Big)=1, \quad
\lim_{k\rightarrow\infty}\E\big[\|\x_i^k - \x_r^k\|^2\big] = 0, \\
&\P\Big(\lim_{k\rightarrow\infty}\|\nabla F(\x_i^k)\| = 0\Big)=1, \quad
\lim_{k\rightarrow\infty}\E\big[\|\nabla F(\x_i^k)\|^2\big] = 0.
\end{align*}
Moreover, if~${\a=\ol{\a}_1}$, \textbf{\texttt{GT-SAGA}} achieves an~$\epsilon$-accurate stationary point in
\begin{align}\label{main0_ncvx}
\mc{O}\left(\frac{EL(F(\ol{\x}^0)-F^*)}{\epsilon} + \frac{\lambda^2(1-\lambda^2)\|\nf(\x^0)\|^2}{n\epsilon}\right)
\end{align}
iterations, where~$E$ is given by 
\begin{align*}
E := \max\left\{\frac{m^{2/3}}{n^{1/3}},1,\frac{\lambda}{(1-\lambda)^2},\frac{\lambda^{1/2}m^{1/2}}{(1-\lambda)^{3/4}}\right\}
\end{align*}
and~$\|\nf(\x^0)\|^2 = \sum_{i=1}^n\|\nabla f_i(\ol{\x}^0)\|^2$.
\end{theorem}

Theorem~\ref{main_ncvx} is formally proved in Subsection~\ref{sec_Th1_proof}. We discuss its implications in the following remarks.
\begin{rmk}[\textbf{Effect of the function heterogeneity}]\label{rmk_hete}
\normalfont We note that $\|\nf(\x^0)\|^2/n$ in the second term of~\eqref{main0_ncvx} can be viewed as a measure of heterogeneity among the local functions. In particular, when all local functions are identical such that ${f_i = f_r = F}, {\forall i,r\in\mc{V}}$, this term diminishes, i.e., it can be shown that $\|\nf(\x^0)\|^2/n = \|\nabla F(\ol{\x}^0)\|^2 \leq 2L(F(\ol{\x}^0) - F^*)$. On the other hand, when the local functions are significantly different, $\|\nf(\x^0)\|^2/n$ can be fairly large compared with $L(F(\ol{\x}^0) - F^*)$. Based on Theorem~\ref{main_ncvx}, it is important to note that the effect of the function heterogeneity $\|\nf(\x^0)\|^2/n$ on the convergence rate of \textbf{\texttt{GT-SAGA}} is decoupled from~$E$, the effect of the local batch size $m$ and the network spectral gap $1-\lambda$. It is further interesting to observe that the heterogeneity effect diminishes when the network is sufficiently either well-connected or weakly-connected. In other words, the function heterogeneity effect is dominated by the network effect in these two extreme cases of interest.     
\end{rmk}

We next view Theorem~\ref{main_ncvx} in two different regimes. 

\begin{rmk}[\textbf{Big-data regime}]\label{rmk_bdata}
\normalfont
We first consider a big-data regime that is often applicable in data centers, where the local batch size~$m$ is relatively large compared with the network spectral gap inverse~$(1-\lambda)^{-1}$ and the number of the nodes~$n$. In particular, if~$m$ large enough such that
\begin{align}\label{bdata_con}
\max\left\{1,\frac{\lambda}{(1-\lambda)^2},\frac{\lambda^{1/2}m^{1/2}}{(1-\lambda)^{3/4}}\right\} \lesssim \frac{m^{2/3}}{n^{1/3}},
\end{align}
Theorem~\ref{main_ncvx} results into an iteration complexity of
\begin{align}\label{bdata}
\mc{O}\left(\frac{m^{2/3}L(F(\ol{\x}^0)-F^*)}{n^{1/3}\epsilon} + \frac{\lambda^2(1-\lambda^2)\|\nf(\x^0)\|^2}{n\epsilon}\right).            
\end{align}

We emphasize that the first term in~\eqref{bdata} matches the iteration complexity of the centralized SAGA with a minibatch size~$n$~\cite{SAGA_reddi}, as~\textbf{\texttt{GT-SAGA}} computes~$n$ component gradients across the nodes in parallel at each iteration.
We note that under the big-data condition~\eqref{bdata_con}, it typically holds that $\|\nf(\x^0)\|^2/n \lesssim m^{2/3}L(F(\ol{\x}^0)-F^*)/n^{1/3}$, i.e., the first term dominates the second term in~\eqref{bdata}. Therefore, \textbf{\texttt{GT-SAGA}} in this regime achieves a non-asymptotic linear speedup, i.e., the total number of component gradient computations required at each node to achieve an $\epsilon$-accurate stationary point is reduced by a factor of~$1/n$, compared with the centralized minibatch SAGA that operates on a single machine.
\end{rmk}

\begin{rmk}[\textbf{Large-scale network regime}]\label{rmk_network}
\normalfont
{\color{black}We now consider the case where a large number of nodes are weakly connected, a scenario that commonly appears in sensor networks, robotic swarms, and ad hoc IoT (Internet of Things) networks. In this case, the number of the nodes~$n$ and the network spectral gap inverse~$(1-\lambda)^{-1}$ are relatively large in comparison with the local batch size~$m$. In particular, if
\begin{align}\label{l_net}
\max\left\{1,\frac{m^{2/3}}{n^{1/3}},\frac{\lambda^{1/2}m^{1/2}}{(1-\lambda)^{3/4}}\right\} \lesssim \frac{\lambda}{(1-\lambda)^2}, 
\end{align}
then the component gradient computation complexity at each node of \textbf{\texttt{GT-SAGA}}, according to Theorem~\ref{main_ncvx}, becomes
\begin{align}\label{saga_ar}
\mc{O}\left(\frac{\lambda L(F(\ol{\x}^0)-F^*)}{(1-\lambda)^2\epsilon} + \frac{\lambda^2(1-\lambda^2)\|\nf(\x^0)\|^2}{n\epsilon}\right).    
\end{align}
We note that the component gradient complexity at each node of~GT-SARAH~\cite{GT-SARAH}, the state-of-the-art decentralized non-convex variance-reduced method, in this regime is
\begin{align}\label{sarah_ar}
\mc{O}\left(\frac{\lambda}{(1-\lambda)^2\epsilon}\left(L(F(\ol{\x}^0)-F^*) + \frac{\|\nf(\x^0)\|^2}{n}\right)\right).    
\end{align}
Comparing~\eqref{sarah_ar} to~\eqref{saga_ar}, we observe that GT-SARAH, unlike \textbf{\texttt{GT-SAGA}}, does not achieve a separation between the dependence of the network spectral gap ${1-\lambda}$ and the function heterogeneity measure~$\|\nf(\x^0)\|^2/n$ on the convergence rate. We hence conclude that \textbf{\texttt{GT-SAGA}} outperforms GT-SARAH if the network is weakly connected and the local functions are largely heterogeneous, i.e., when ${1-\lambda}$ is small and~$\|\nf(\x^0)\|^2/n$ is large. Moreover, we recall from Remark~\ref{rmk_prac} that \textbf{\texttt{GT-SAGA}} is single-timescale and thus is much easier to implement than the two-timescale GT-SARAH over large-scale networks.} We also emphasize that the storage requirement of \textbf{\texttt{GT-SAGA}} in this regime is significantly relaxed since the data samples are distributed across a large network, leading to a small local batch size~$m$ at each node.
\end{rmk}

\subsection{Global PL condition}
\begin{theorem}\label{main_PL}
Let Assumptions~\ref{sample},~\ref{smooth},~\ref{PL}, and~\ref{network} hold. If the step-size~$\a$ of \textbf{\texttt{GT-SAGA}} satisfies~$0<\a\leq\ol{\a}_2$, where
\begin{align*}
\ol{\a}_2 := \min\bigg\{&\frac{(1-\lambda^2)^2}{55\lambda L}, \frac{1-\lambda^2}{13\lambda\kappa^{1/4}L},\frac{(1-\lambda^2)^3}{388\lambda^2nL}, \\
&\frac{n^{1/3}}{10.5m^{2/3}\kappa^{1/3}L},\frac{1}{36L},\frac{1-\lambda^2}{2\mu},\frac{1}{4m\mu}\bigg\},    
\end{align*}
then all nodes converge linearly at the rate~$\mc{O}((1-\mu\a)^k)$ to a global minimizer of~$F$. In particular, if ${\a =\ol{\a}_2}$, then all nodes agree on an~$\epsilon$-accurate global minimizer in $$\mc{O}\left(\max\Big\{Q_{\text{opt}},Q_{\text{net}}\Big\}\log\frac{1}{\epsilon}\right)$$ iterations, where~$Q_{\text{opt}}$ and~$Q_{\text{net}}$ are given respectively by
\begin{align*}
&Q_{\text{opt}} := \max\left\{\frac{m^{2/3}\kappa^{4/3}}{n^{1/3}},\kappa,m\right\},
\\
&Q_{\text{net}} := \max\left\{\frac{\lambda\kappa}{(1-\lambda)^2},\frac{\lambda\kappa^{5/4}}{1-\lambda},\frac{\lambda^2n\kappa}{(1-\lambda)^3},\frac{1}{1-\lambda}\right\}.
\end{align*}
\end{theorem}
Theorem~\ref{main_PL} is formally proved in Subsection~\ref{PL_sec}. 
The following remarks discuss a few key aspects of it.

\begin{rmk}[\textbf{Linear rate under the global PL condition}]\label{rmk_pl_bdata}
\normalfont Theorem~\ref{main_PL} shows that \textbf{\texttt{GT-SAGA}} linearly converges to an optimal solution when the global~$F$ additionally satisfies the PL condition. This is the first linear rate result for decentralized variance-reduced methods under the PL condition while the existing ones require strong convexity, e.g.,~\cite{GTVR,DAVRG,DSA,AccDVR,Network-DANE}. 
A notable feature of the linear rate in Theorem~\ref{main_PL} is that the effects of the local batch size $m$ and the network spectral gap $1-\lambda$ are decoupled. Hence, in a big-data regime where the local batch size $m$ is sufficiently large such that $Q_{\text{net}}\lesssim Q_{\text{opt}}$, $\textbf{\texttt{GT-SAGA}}$ achieves a network topology-independent rate of $\mc{O}(Q_{\text{opt}}\log\frac{1}{\epsilon})$. In addition, we note that
Theorem~\ref{main_PL} implies the linear rate of \textbf{\texttt{GT-SAGA}} in the almost sure sense under the PL condition, by Chebyshev's inequality and the Borel-Cantelli lemma; see Lemma 7 in~\cite{GTVR} for details.
\end{rmk}

\begin{rmk}[\textbf{Comparison with other decentralized gradient methods}]\label{rmk_pl_full}
\normalfont  When the local batch size~$m$ is relatively large, the linear rate of \textbf{\texttt{GT-SAGA}} improves that of the existing decentralized batch gradient methods~\cite{GT_zero,PL_DPD,improved_DSGT_Xin} under the PL condition in terms of the component gradient computation complexity. Moreover, decentralized online stochastic gradient methods, e.g.,~\cite{improved_DSGT_Xin,PD_SGD}, only exhibit sublinear rate under the PL condition due to the persistent variances of the stochastic gradients. Therefore, \textbf{\texttt{GT-SAGA}} achieves faster convergence under the PL condition compared with the existing decentralized methods, demonstrating the advantage of the employed SAGA variance reduction scheme that is able to exploit the finite-sum structure of local functions. 
\end{rmk}

\begin{rmk}[\textbf{Improved convergence results for the centralized minibatch SAGA}]
\label{rmk_pl_c}
\normalfont {\color{black}When~$\lambda = 0$, i.e., when the underlying network is a complete graph whose weight matrix can be easily chosen as~$\ul{\W}= \frac{1}{n}\mb{1}_n\mb{1}_n^\top$,} \textbf{\texttt{GT-SAGA}} reduces to the centralized minibatch SAGA and achieves the linear rate of~$\mc{O}(Q_{\text{opt}}\log\frac{1}{\epsilon})$. Hence, a special case of Theorem~\ref{main_PL}, i.e,~$\lambda = 0$, provides the first linear rate result under the PL condition for the centralized SAGA. Indeed, the existing linear rate results~\cite{SAGA_reddi,prox_SAGA_reddi} under the PL condition are only applicable to a modified SAGA that \emph{periodically restarts}~$\mc{O}(\log\frac{1}{\epsilon})$ times with the output of each cycle being selected randomly from the past iterates in this cycle. This procedure is not feasible particularly in decentralized settings. In contrast, the linear rate shown Theorem~\ref{main_PL} is on the \emph{last iterate} of the original SAGA without periodic restarting and sampling. 
\end{rmk}

\section{Numerical experiments} \label{exp}
\begin{figure*}
\centering
\includegraphics[width=1.753in]{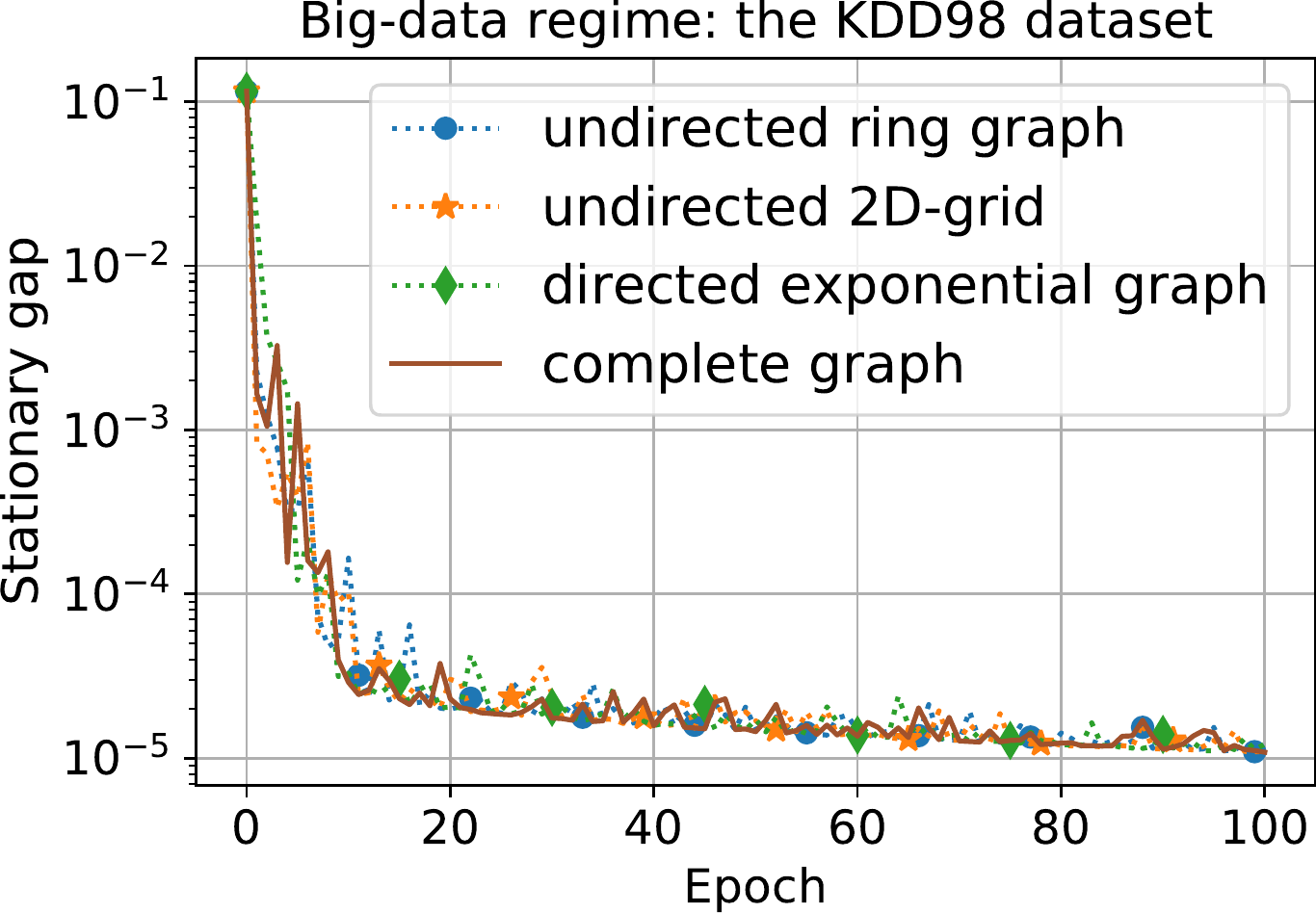}
\includegraphics[width=1.753in]{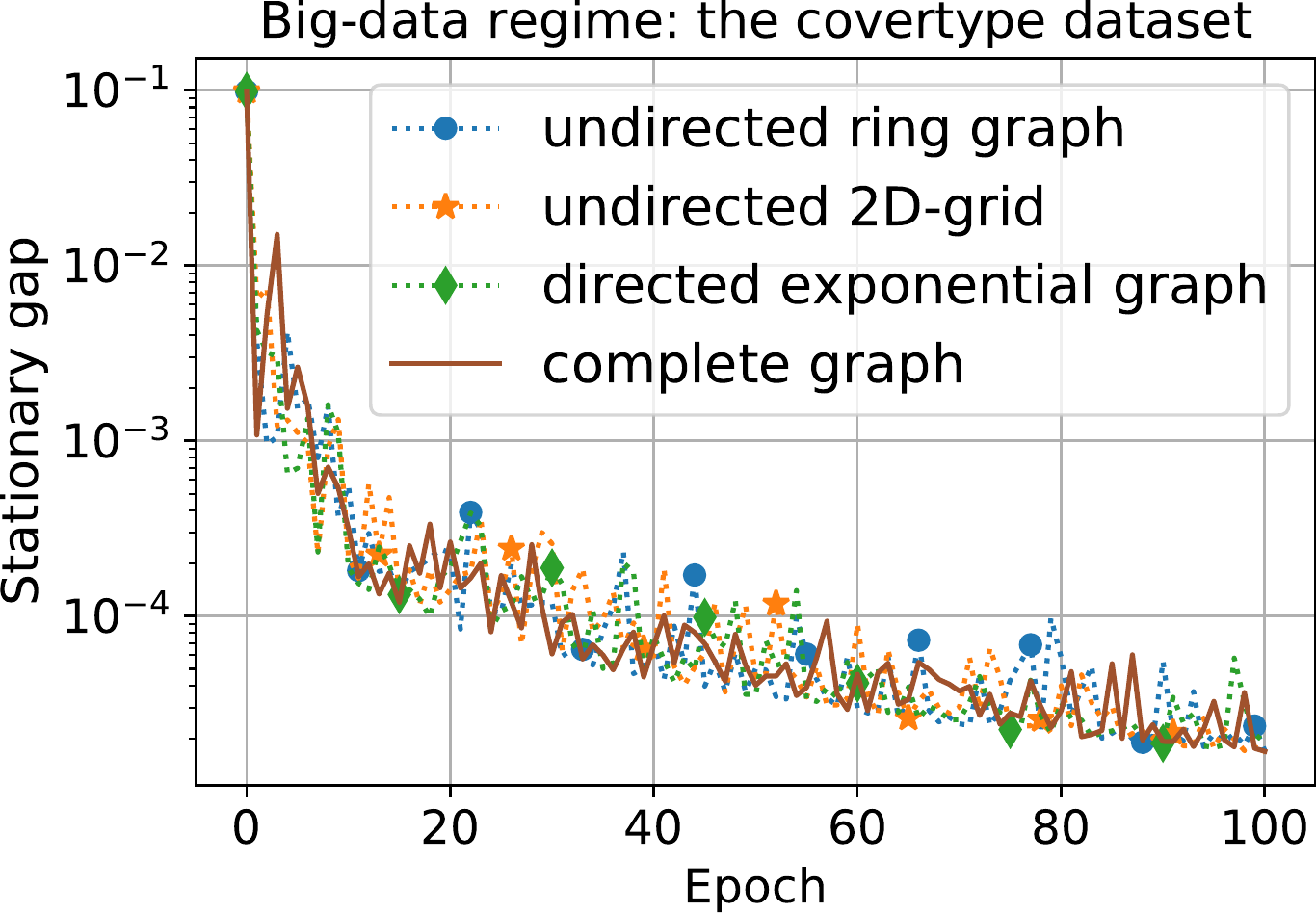}
\includegraphics[width=1.753in]{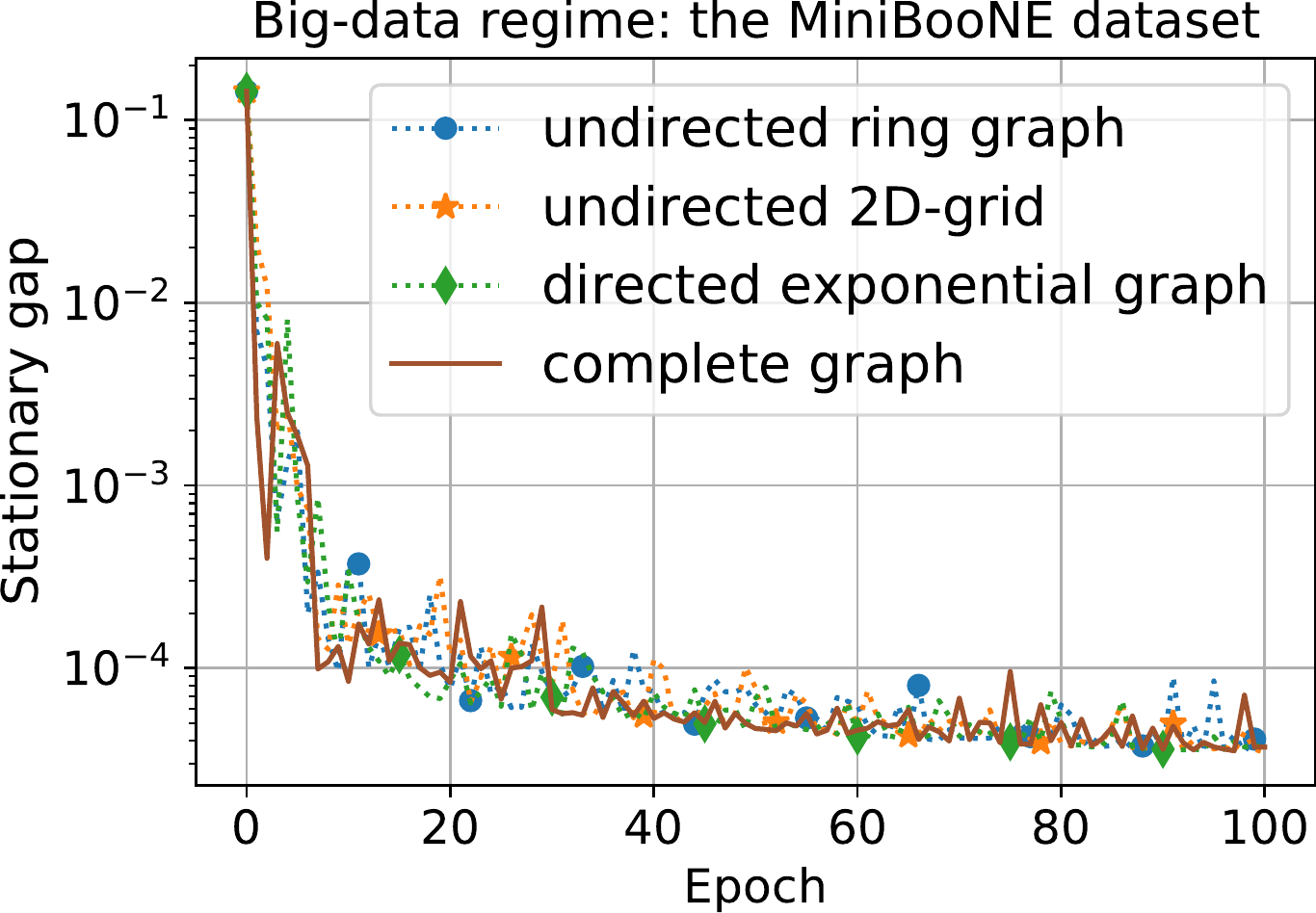}
\includegraphics[width=1.753in]{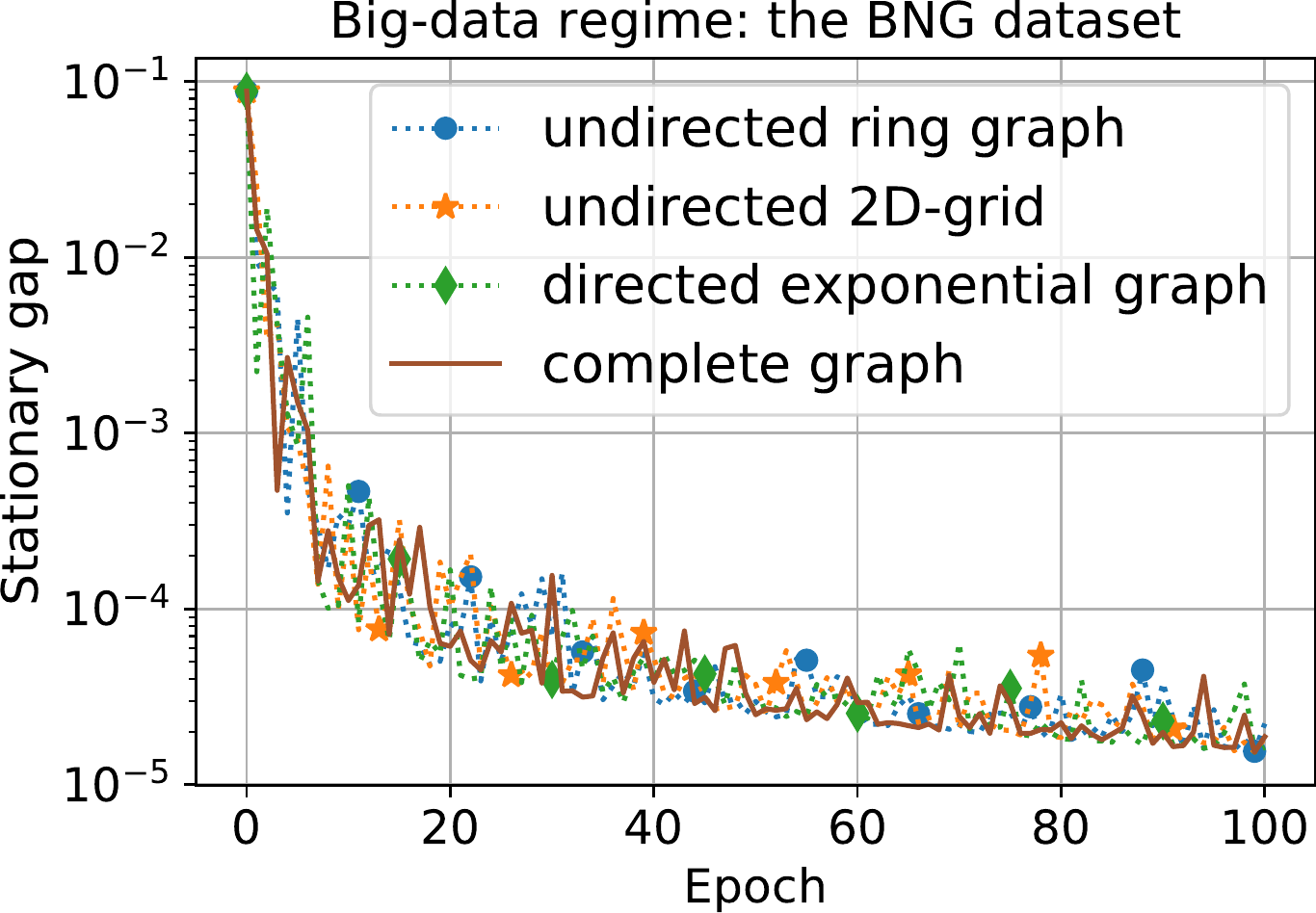}
\caption{Big data regime: the network topology-independent convergence rate of \textbf{\texttt{GT-SAGA}} on the KDD98, covertype, MiniBooNE, and BNG(sonar) datasets.}
\label{bdata_simul}
\end{figure*}

\begin{figure*}
\centering
\includegraphics[width=1.753in]{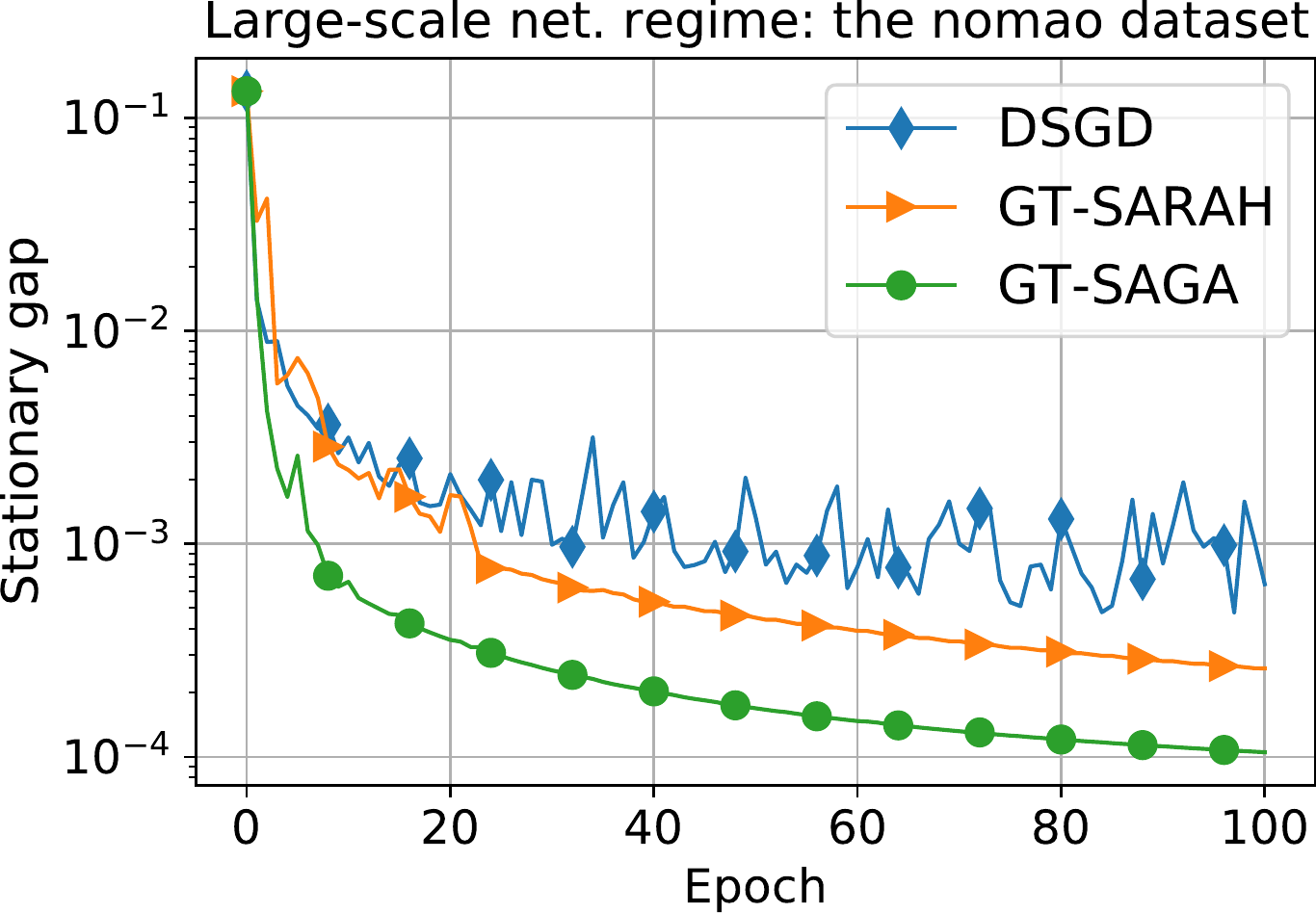}
\includegraphics[width=1.753in]{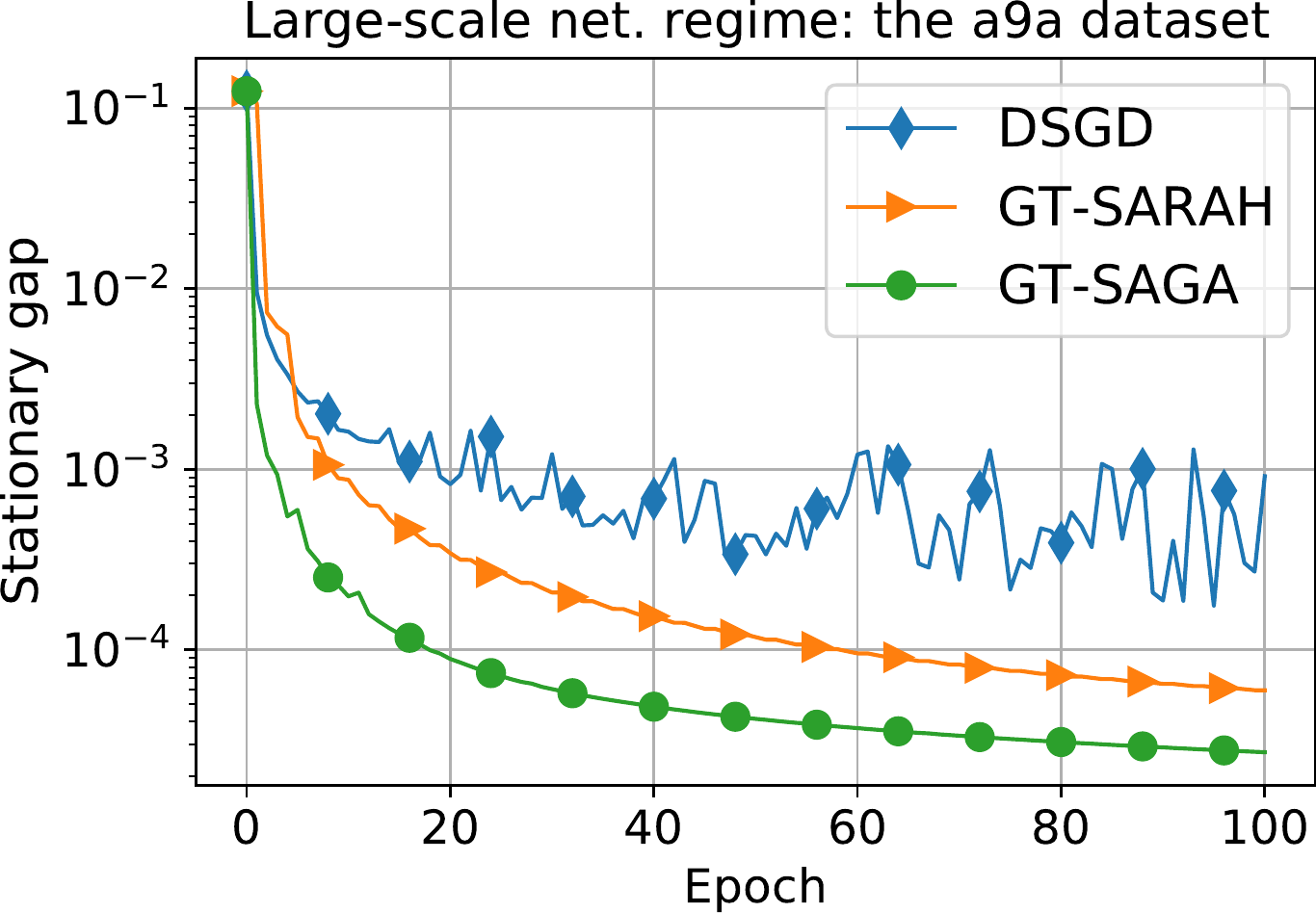}
\includegraphics[width=1.753in]{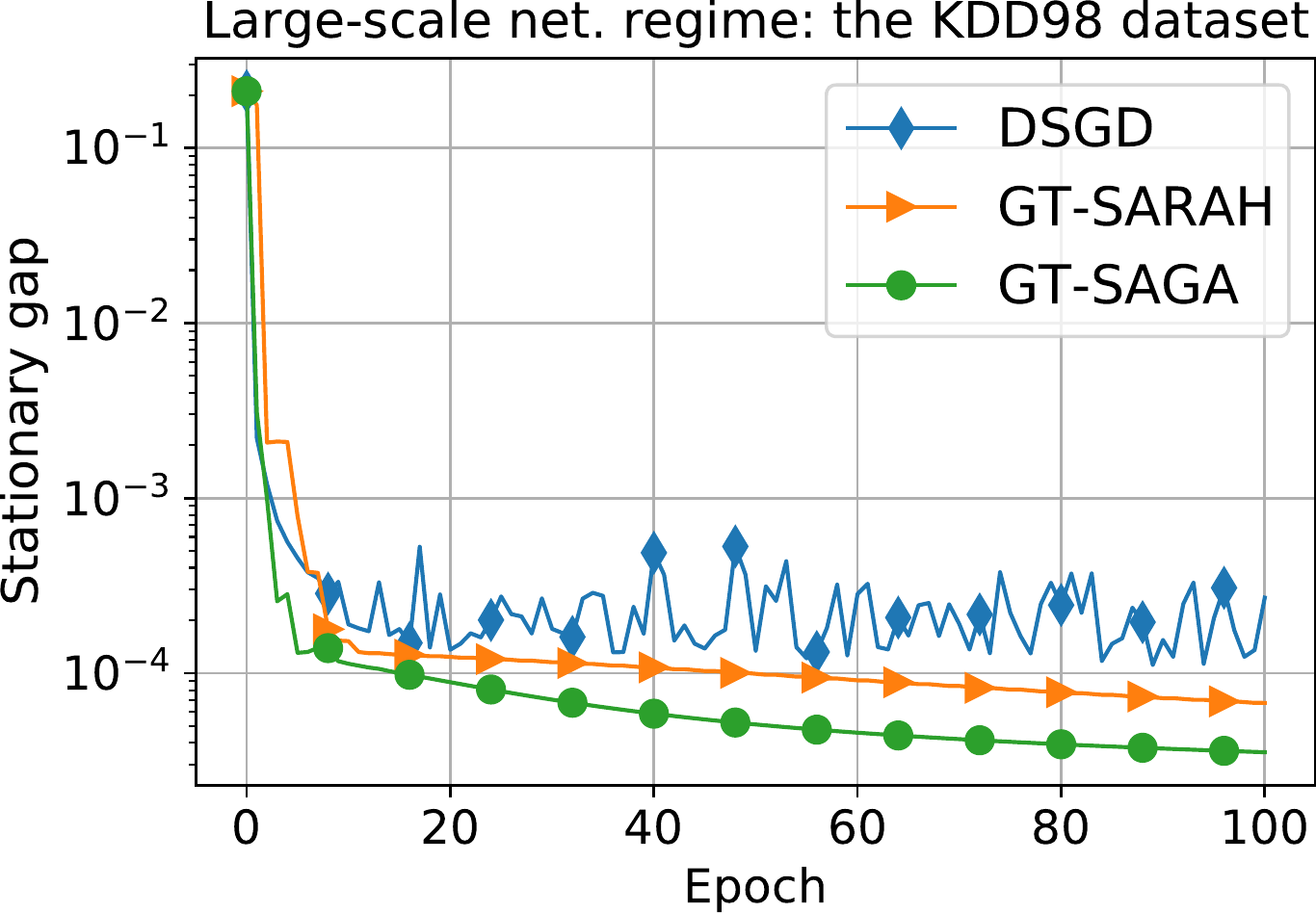}
\includegraphics[width=1.753in]{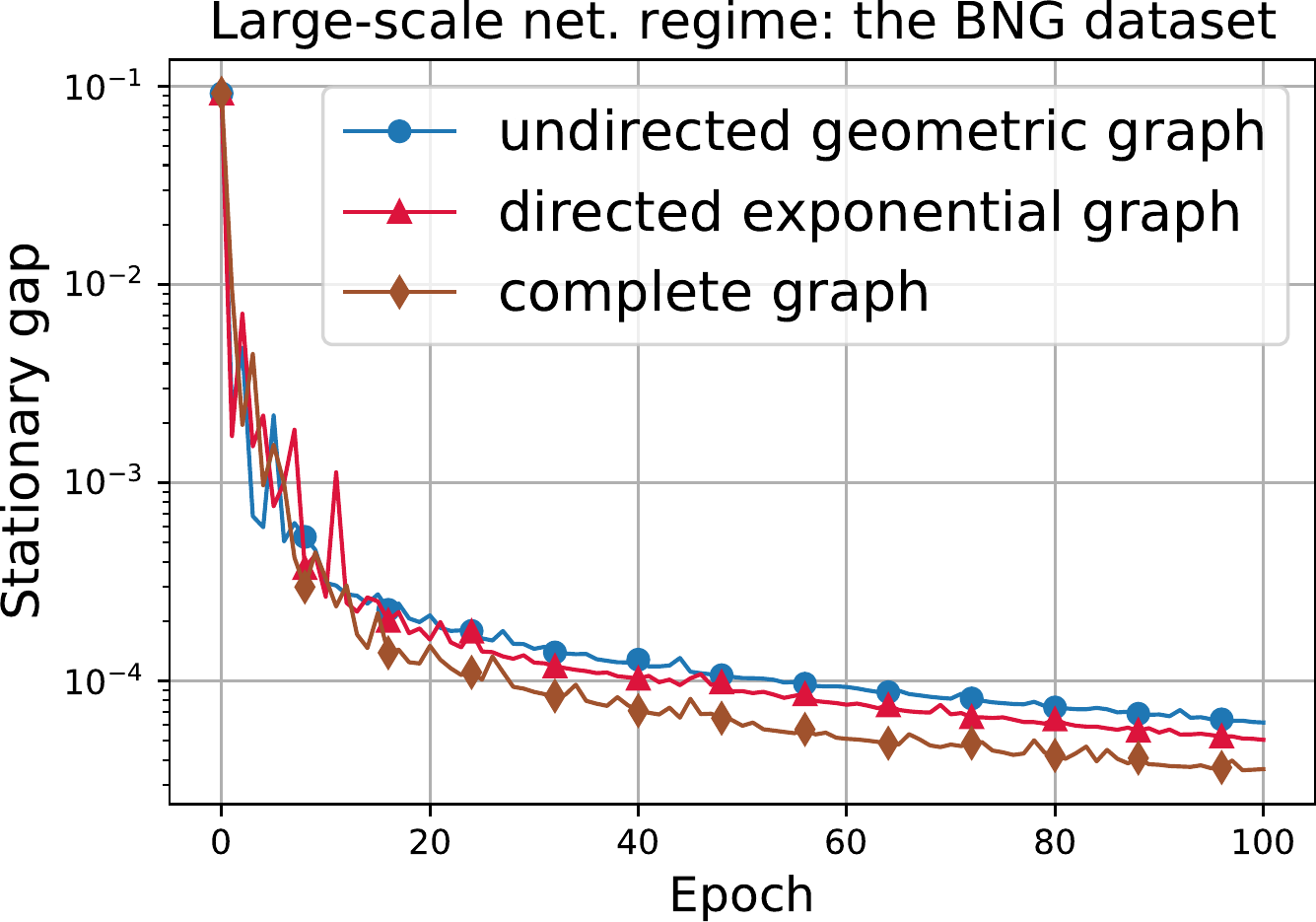}
\caption{Large-scale network regime: (i) the first three plots present the performance comparison between \textbf{\texttt{GT-SAGA}}, DSGD, and GT-SARAH on the nomao, a9a, and KDD98 datasets; (ii) the last plot presents the performance of \textbf{\texttt{GT-SAGA}} over different graph topologies in this regime on the BNG(sonar) dataset.}
\label{lnet_simul}
\end{figure*}

\begin{figure}
\centering
\includegraphics[width=1.722in]{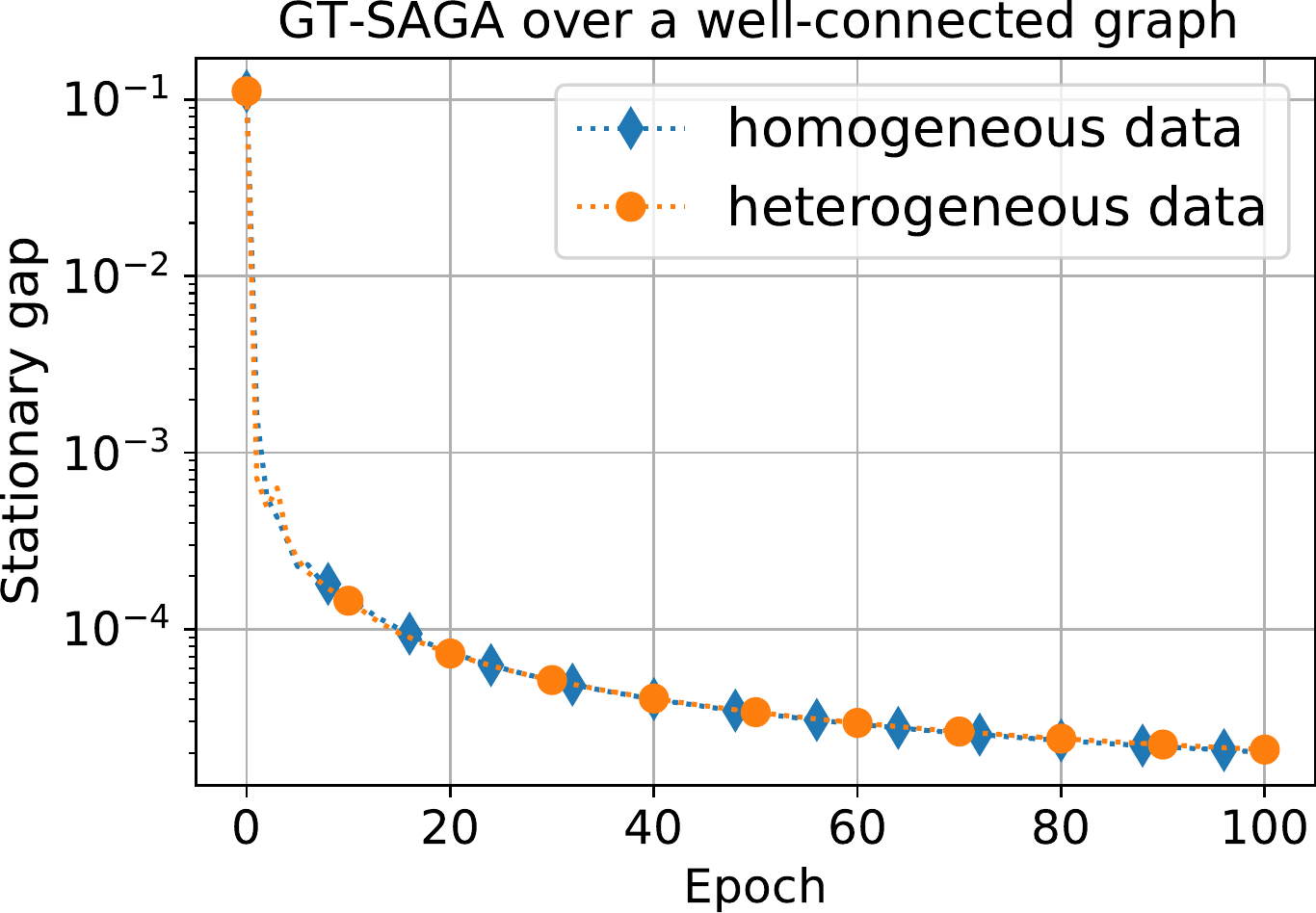}
\includegraphics[width=1.722in]{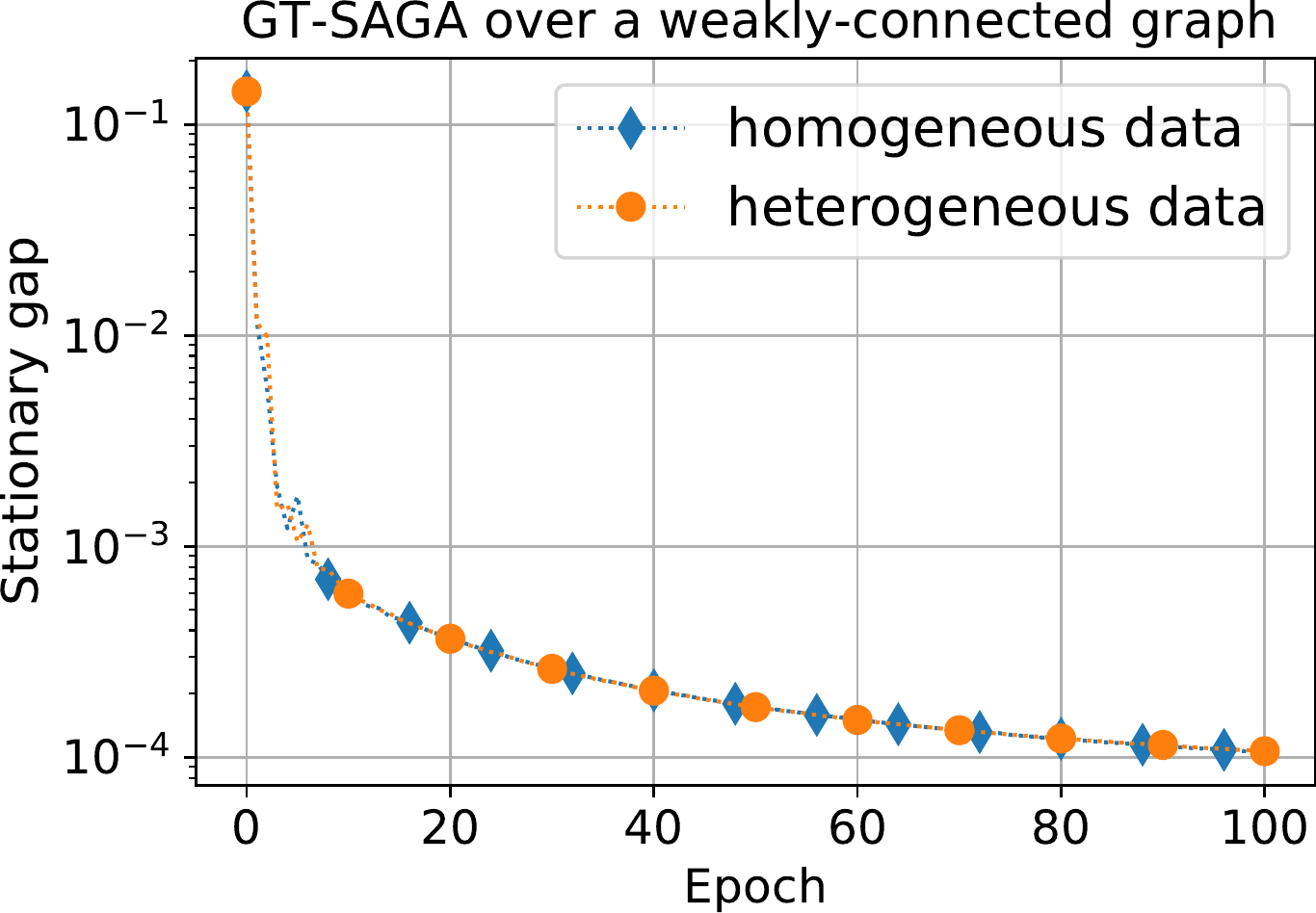}
\caption{Robustness of \textbf{\texttt{GT-SAGA}} to heterogeneous data over well- and weakly-connected graphs on the nomao dataset.}
\label{robust_simul}
\end{figure}

\begin{figure*}
\centering
\includegraphics[width=1.8in]{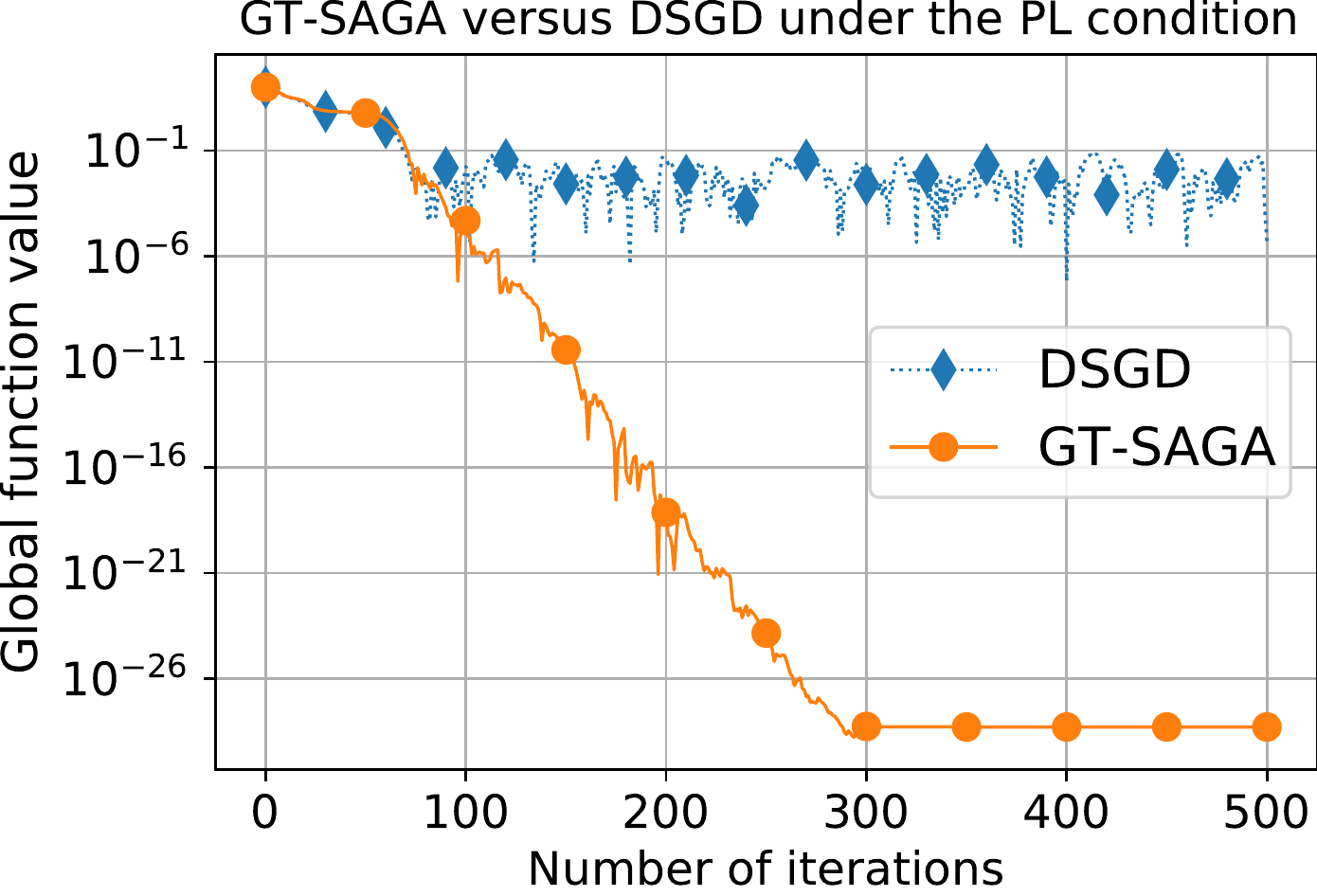}
\includegraphics[width=1.72in]{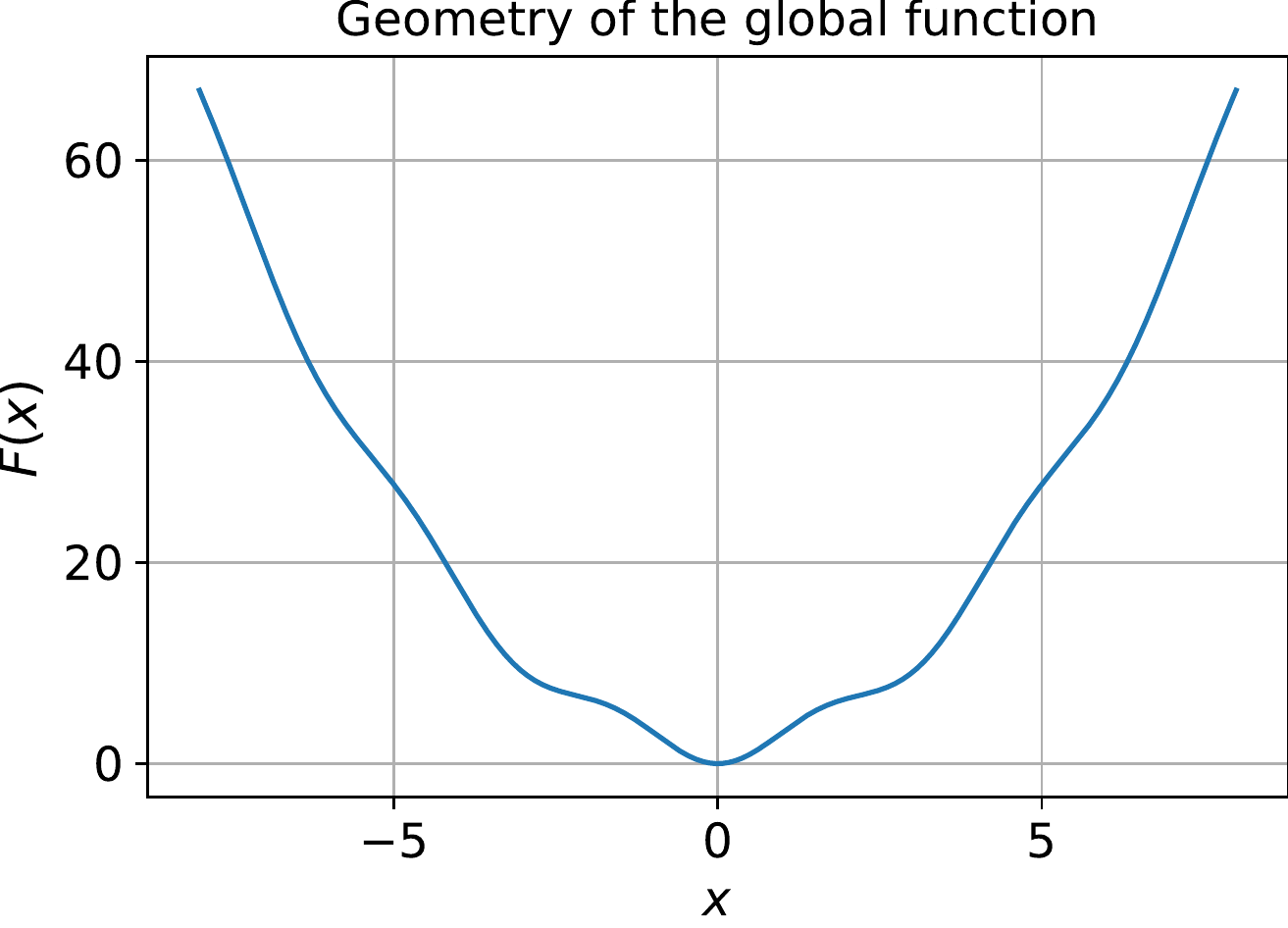}
\includegraphics[width=1.75in]{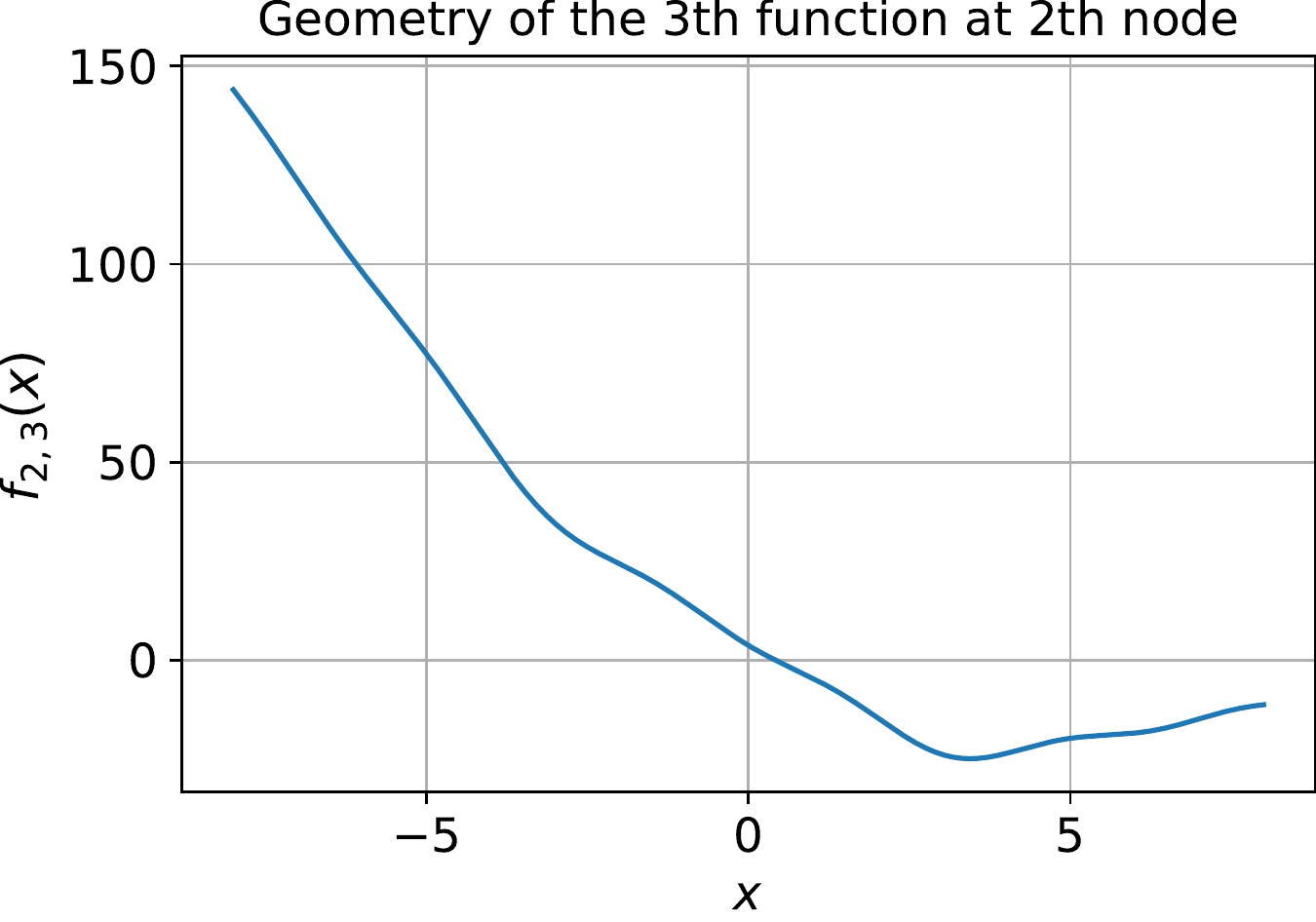}
\includegraphics[width=1.72in]{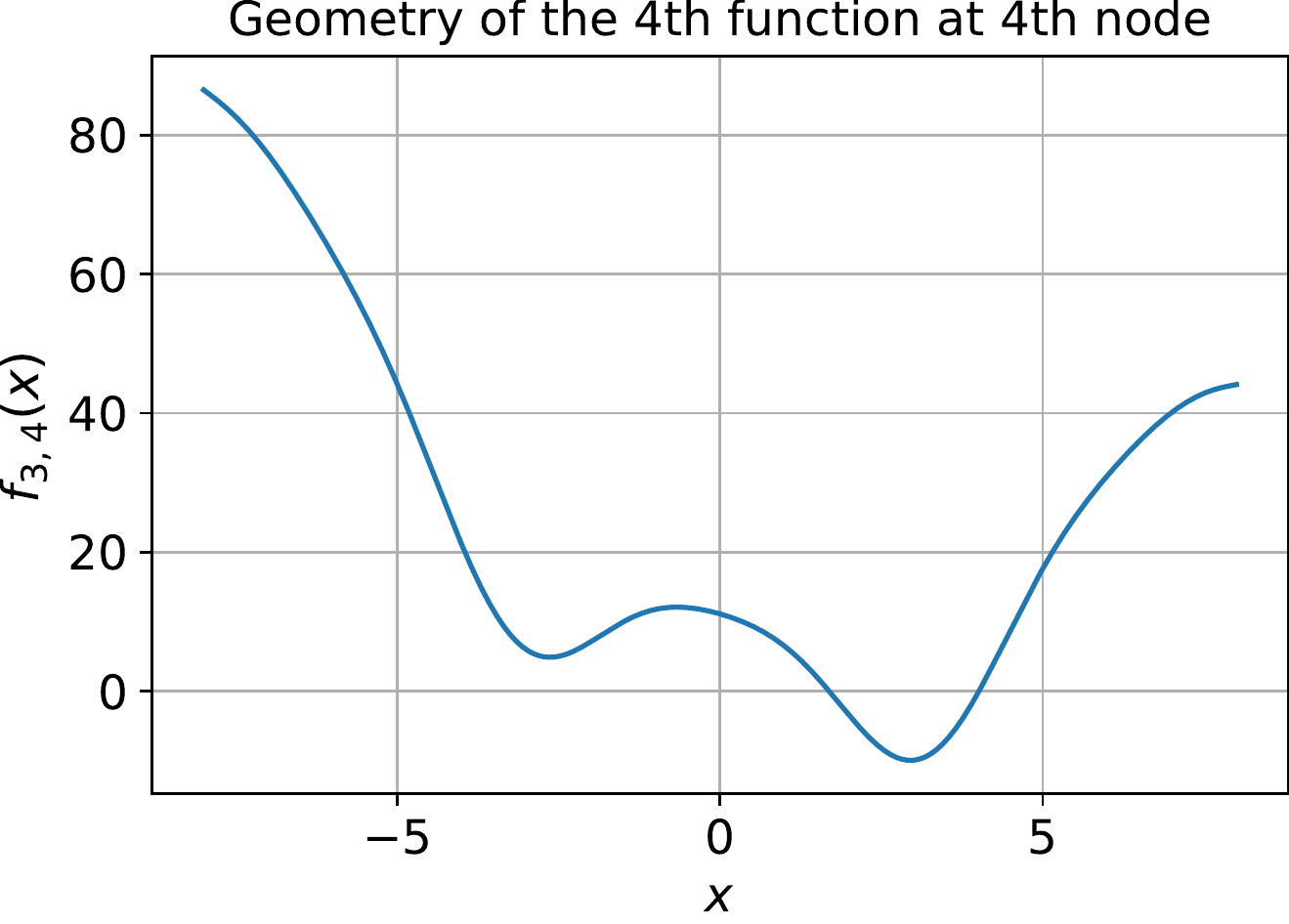}
\caption{The PL condition: (i) the first plot presents the performance comparison between \textbf{\texttt{GT-SAGA}} and DSGD when the global function satisfies the PL condition; (ii) the last three plots present the geometry comparison of the global and local component functions.}
\label{PL_simul}
\end{figure*}

In this section, we present numerical simulations to illustrate our main theoretical results. The network topologies of interest are undirected ring, undirected 2D-grid, directed exponential, undirected geometric, and complete graphs; see~\cite{GTVR,tutorial_nedich,SGP_ICML} for details of these graphs. The doubly stochastic weights are set to be equal for the ring and exponential graphs, and are generated by the lazy Metropolis rule for the grid and geometric graphs. We manually optimize the parameters of all algorithms in all experiments for their best performance.

\subsection{Non-convex binary classification}\label{exp_BC}
{\color{black}In this subsection, we consider a decentralized non-convex generalized linear model for binary classification. In view of Problem~\eqref{Problem}, each component cost~$f_{i,j}$ is defined as~\cite{nonconvex_BC}
\begin{align*}
f_{i,j}(\x) := \ell\big(\xi_{i,j}\mb{x}^\top\bds{\theta}_{i,j}\big),
~~
\ell(u) := \left(1-\frac{1}{1+\exp(-u)}\right)^2,
\end{align*}
where~${\bds{\theta}_{i,j}\in\mathbb{R}^{p}}$ is the~$j$-th data vector at the~$i$-th node, $\xi_{i,j}\in\{-1,+1\}$ is the label of~$\bds{\theta}_{i,j}$, and~$\ell:\mathbb{R}\rightarrow\mathbb{R}$ is a $\frac{4}{3}$-smooth non-convex loss. We normalize each data to be~$\|\bds{\theta}_{i,j}\| = 1, \forall i,j$. Since $\nabla^2f_{i,j}(\x) = \ell''(\xi_{i,j}\x^\top\bds{\theta}_{i,j})\bds{\theta}_{i,j}\bds{\theta}_{i,j}^\top$, it can be verified that $\|\nabla^2f_{i,j}(\x)\|=|\ell''(\xi_{i,j}\x^\top\bds{\theta}_{i,j})|\leq\frac{4}{3}$. Hence each component cost~$f_{i,j}$ is non-convex and~$\frac{4}{3}$-smooth. We measure the performance of the algorithms in question in terms of the decrease of the stationary gap~$\|\nabla F(\ol{\x})\|$ versus epochs, where~$\ol{\x} := \frac{1}{n}\sum_{i=1}^n\x_i$ for~$\x_i$ being the model at node~$i$ and each epoch represents~$m$ component gradient evaluations at each node. All nodes start from a vector randomly generated from the standard Gaussian distribution.} The statistics of the datasets used in the experiments are provided in Table~\ref{datasets}.

\begin{table}[hbt]
\renewcommand{\arraystretch}{1.1}
\caption{Datasets used in numerical experiments, available at \href{https://www.openml.org/}{https://www.openml.org/}.}
\vspace{-0.2cm}
\begin{center}
\begin{tabular}{|c|c|c|c|}
\hline
\textbf{Dataset} & \textbf{train} ($N = nm$)  & \textbf{dimension} ($p$) \\ \hline
nomao & $30,\!000$ & $119$  \\ \hline
a9a & $48,\!800$ & $124$  \\ \hline
w8a & $60,\!000$ & $300$  \\ \hline
KDD98 & $80,\!000$ & $478$  \\ \hline
covertype & $100,\!000$ & $55$ \\ \hline
MiniBooNE & $100,\!000$ & $51$  \\ \hline
BNG(sonar) & $100,\!000$ & $61$  \\ \hline
\end{tabular}
\end{center}
\label{datasets}
\end{table}

\subsubsection{Big data regime} We first test the convergence behavior of \textbf{\texttt{GT-SAGA}} in the big data regime by uniformly distributing the KDD98, covertype, MiniBooNE, and BNG(sonar) datasets over a network of~${n=20}$ nodes. We consider four different network topologies with decreasing sparsity, i.e., the undirected ring, undirected 2D-grid, directed exponential, and complete graph; their corresponding {\color{black}second largest singular values} of the weight matrices are~${\lambda = 0.98, 0.97, 0.6,0}$, respectively. It can be verified that the big data condition~\eqref{bdata_con} holds. The experimental results are shown in Fig.~\ref{bdata_simul}, where we observe that the convergence rate of \textbf{\texttt{GT-SAGA}} is independent of the network topology in this big data regime; see Remark~\ref{rmk_bdata}.

\subsubsection{Large-scale network regime} We next compare the performance of \textbf{\texttt{GT-SAGA}} with DSGD~\cite{DSGD_NIPS} and GT-SARAH~\cite{GT-SARAH} in the large-scale network regime. To this aim, we generate a sparse geometric graph of~${n=200}$ nodes with $\lambda \approx 0.99$ and uniformly distribute the nomao, a9a, w8a, and BNG(sonar) datasets over the nodes. It can be verified that the large-scale network condition~\eqref{l_net} holds. {\color{black}The numerical results are presented in Fig.~\ref{lnet_simul}: the first three plots show that \textbf{\texttt{GT-SAGA}} achieves the best performance among the algorithms in comparison, while the last plot shows that the convergence rate of \textbf{\texttt{GT-SAGA}} is dependent on the network topology in this large-scale network regime; see Remark~\ref{rmk_network}.}

\subsubsection{Robustness to heterogeneous data}\label{exp_rob}
{\color{black}We now make the data distributions across the nodes significantly heterogeneous by letting each node only have data samples of one label, so that no node can train a valid classification model only from its local data. 
We compare the performance of \textbf{\texttt{GT-SAGA}} under heterogeneous and homogeneous distribution of the nomao dataset.
We consider a well-connected graph, i.e., the~$20$-node exponential graph,
and a weakly-connected graph, i.e.,
the~$200$-node geometric graph. The numerical results are shown in Fig.~\ref{robust_simul}, where we observe that the convergence rate of \textbf{\texttt{GT-SAGA}}
is not affected by the data heterogeneity over both graphs; see Remark~\ref{rmk_hete}.}

\subsection{Synthetic functions that satisfy the PL condition}
{\color{black}Finally, we verify the linear rate of \textbf{\texttt{GT-SAGA}} when the global function~$F$ satisfies the PL condition. Specifically, we choose each component function~$f_{i,j}:\mathbb{R}\rightarrow\mathbb{R}$ as~$$f_{i,j}(x) = x^2 + 3\sin^2(x) + a_{i,j}\cos(x) + b_{i,j}x,$$ where~${\sum_{i=1}^n\sum_{j=1}^ma_{i,j} = 0}$ and ${\sum_{i=1}^n\sum_{j=1}^mb_{i,j} = 0}$ such that $a_{i,j}\neq0$,~$b_{i,j}\neq0$,~$\forall i,j$. This formulation hence leads to the global function~${F(x) = x^2 + 3\sin^2(x)}$. It can be verified that~$F$ is non-convex and satisfies the PL condition~\cite{PL_1}. Note that each~$f_{i,j}$ is nonlinear and highly deviated from~$F$; see the last three plots in Fig.~\ref{PL_simul} for a comparison of local and global geometries. We use the~$20$-node exponential graph and set~$m = 5$. It can be observed from the first plot in Fig.~\ref{PL_simul} that \textbf{\texttt{GT-SAGA}} achieves linear rate to the optimal solution, while DSGD converges to an inexact solution; see Remark~\ref{rmk_pl_full}.}

\section{Convergence analysis}\label{conv}
In this section, we present the convergence analysis of \textbf{\texttt{GT-SAGA}}, i.e., the sublinear convergence for general smooth non-convex functions and the linear convergence when the global function~$F$ additionally satisfies the PL condition. Throughout this section, we assume Assumption~\ref{sample},~\ref{smooth}, and~\ref{network} hold without explicitly stating them; we only assume Assumption~\ref{PL} hold in Subsection~\ref{PL_sec}. In Subsections~\ref{sec_bvr}-\ref{sec_sgt}, we establish key relationships between several important quantities, based on which the proofs of Theorem~\ref{main_ncvx} and~\ref{main_PL} are derived in Subsections~\ref{sec_Th1_proof} and~\ref{PL_sec} respectively. 
We start by presenting some preliminary facts. 
\subsection{Preliminaries}
\textbf{\texttt{GT-SAGA}} can be written in the following form:~$\forall k\geq0$,
\begin{subequations}
\begin{align}
\mb{y}^{k+1} &= \W\left(\mb{y}^{k} + \mb{g}^{k} - \mb{g}^{k-1}\right), \label{gtsaga_y}\\ 
\mb{x}^{k+1} &= \W\left(\mb{x}^{k} - \alpha\mb{y}^{k+1}\right),     \label{gtsaga_x}
\end{align}
\end{subequations}
where~$\mb x^k, \mb y^k, \mb g^k$ are random vectors in~$\mathbb{R}^{np}$ that concatenate all local states~$\{\mb{x}_i^k\}_{i=1}^n$, gradient trackers~$\{\mb{y}_i^k\}_{i=1}^n$, local SAGA estimators~$\{\mb{g}_i^k\}_{i=1}^n$, respectively, and~$\W = \ul{\W}\otimes \I_p.$  We denote~$\F^k$ as the filtration of \textbf{\texttt{GT-SAGA}}, i.e.,~$\forall k\geq1$,
\begin{align*}
\F^k := \sigma\left(\left\{\tau_i^t,s_i^t:i\in\mc{V},t\leq k-1\right\}\right),
\qquad
\F^0 :=\left\{\phi,\Omega\right\},
\end{align*} 
where~$\phi$ is the empty set. It can be verified that $\mb{x}^k$, $\mb{y}^k$ and $\mb{z}_{i,j}^k,\forall i,j$, are~$\F^k$-measurable and~$\mb{g}^k$ is~$\mc{F}^{k+1}$-measurable for all~$k\geq0$.   
We use~$\mathbb{E}[\cdot|\mc{F}^k]$ to denote the conditional expectation  with respect to~$\mc{F}^k$. For the ease of exposition, we introduce the following quantities:
\begin{align*}
&{\color{black}\mb{J} := \left(\mb{1}_n\mb{1}_n^\top/n\right)\otimes\mb{I}_p,} \\
&\nabla\mb{f}(\mb{x}^k) = \big[\nabla f_1(\mb{x}_1^k)^\top,\cdots,\nabla f_n(\mb{x}_n^k)^\top\big]^\top, \\
&\ol{\nabla\mb{f}}(\mb{x}^k) = \left(\mb{1}_n^\top\otimes \I_p/n\right)\nabla\mb{f}(\mb{x}^k), \qquad\ol{\mb x}^k = \left(\mb{1}_n^\top\otimes \I_p/n\right)\mb{x}^k, \\
&\ol{\mb y}^k = \left(\mb{1}_n^\top\otimes \I_p/n\right)\mb{y}^k, \qquad\ol{\mb g}^k = \left(\mb{1}_n^\top\otimes \I_p/n\right)\mb{g}^k.
\end{align*}
We assume $\ol{\mb x}^0\in\mbb{R}^p$ is constant and hence all random variables generated by \textbf{\texttt{GT-SAGA}} have bounded second moment. The following lemma lists several well-known facts in the context of gradient tracking and SAGA estimators, which may be found in~\cite{harnessing,SAGA,GTVR,polyak1987introduction,NEXT_scutari}.
\begin{lemma}\label{basic}
The following relationships hold.
\begin{enumerate}[(a)]
\item $\forall\mb{x}\in\mbb{R}^{np}$,~$\|\W\mb{x}-\J\mb{x}\|\leq\lambda\|\mb{x}-\J\mb{x}\|$. \label{W}
\item $\ol{\mb y}^{k+1} = \ol{\mb g}^k ,\forall k\geq0$. \label{track}
\item $\|\ol{\nabla\mb{f}}(\mb{x}^k)-\nabla F(\ol{\mb{x}}^k)\|^2\leq\frac{L^2}{n}\|\mb{x}^k-\J\mb{x}^k\|^2,\forall k\geq0$. \label{Lbound}
\item $\E[\mb{g}_i^k|\F^k] = \nabla f_i(\mb{x}_i^k),\forall i\in\mc{V},\forall k\geq0$.
\item $\|\nabla F(\x)\|^2 \leq 2L\left(F(\x)-F^*\right).$ \label{upper_L}
\end{enumerate}
\end{lemma}
Note that Lemma~\ref{basic}(\ref{upper_L}) is a consequence of the~$L$-smoothness of the global function~$F$ and is only used in Subsection~\ref{PL_sec} while other statements in Lemma~\ref{basic} are frequently utilized throughout the analysis. 
The next lemma states some standard inequalities on the network consensus error~\cite{GTVR,improved_DSGT_Xin}.
\begin{lemma}
The following inequality holds:~$k\geq0$,
\begin{align}
\|\mb{x}^{k+1}-\J\mb{x}^{k+1}\|^2 \leq&~\textstyle\frac{1+\lambda^2}{2} \|\mb{x}^{k}-\J\mb{x}^{k}\|^2 \nonumber\\
&\textstyle\quad+ \frac{2\alpha^2\lambda^2}{1-\lambda^2}\|\mb{y}^{k+1}-\J\mb{y}^{k+1}\|^2. \label{consensus1}\\
\|\mb{x}^{k+1}-\J\mb{x}^{k+1}\|^2 \leq&~2\lambda^2\|\mb{x}^{k}-\J\mb{x}^{k}\|^2 \nonumber\\
&\quad+ 2\alpha^2\lambda^2\|\mb{y}^{k+1}-\J\mb{y}^{k+1}\|^2. \label{consensus2}
\\
\|\mb{x}^{k+1}-\J\mb{x}^{k+1}\| \leq&~\lambda\|\mb{x}^{k}-\J\mb{x}^{k}\|
+ \alpha\lambda\|\mb{y}^{k+1}-\J\mb{y}^{k+1}\|. \label{consensus3}
\end{align}
\end{lemma}

\subsection{Bounds on the variance of local SAGA estimators}\label{sec_bvr}
 In this subsection, we bound the variance of the local SAGA gradient estimators~$\mb{g}_i^k$'s. For analysis purposes, we construct two auxiliary~$\F^k$-adapted sequences: $\forall i\in\mc{V}$, $\forall k\geq0$,
\begin{align*}
t_i^k := \frac{1}{m}\sum_{j=1}^m\|\ol{\mb{x}}^k - \mb{z}_{i,j}^k\|^2, \qquad
t^k := \frac{1}{n}\sum_{i=1}^n t_i^k.
\end{align*}
These two sequences are essential in the convergence analysis. 
We note that~$t^k$ measures the average distance between the mean state~$\ol{\x}^k$ of the networked nodes and the latest iterates~$\mb{z}_{i,j}^k$'s where the component gradients were computed at iteration~$k$ in the gradient tables.
Intuitively,~$t^k$ goes to~$0$ as all nodes in \textbf{\texttt{GT-SAGA}} reach consensus on a stationary point. We will establish a contraction argument in~$t^k$ in Subsection~\ref{sec_t}. In the following lemma, we show that the variance of~$\mb{g}_i^k$ may be bounded by the network consensus error and~$t^k$. 

\begin{lemma}\label{BVR}
The following inequality holds:~$\forall k\geq0$,
\begin{align}
&\mathbb{E}\big[\|\mb{g}^k - \nabla\mb{f}(\mb{x}^k)\|^2|\mc{F}^k\big]
\leq 2L^2\|\mb{x}^k - \J\mb{x}^k\|^2
+ 2nL^2 t^k,         \label{BVR1}                 \\
&\mathbb{E}\big[\|\ol{\mb g}^k\|^2|\mc{F}^k\big] 
\leq
\frac{2L^2}{n^2}\|\mb{x}^k - \J\mb{x}^k\|^2
+\frac{2L^2}{n}t^k
+\|\ol{\nabla\mb{f}}(\mb{x}^k)\|^2. \label{BVR2}
\end{align}
\end{lemma}
\begin{proof}
We denote~$\wh{\nabla}_i^k: = \nabla f_{i,\tau_i^{k}}\big(\mb{x}_{i}^{k}\big) - \nabla f_{i,\tau_i^{k}}\big(\mb{z}_{i,\tau_i^{k}}^{k}\big)$, $\forall i\in\mc{V}$, $\forall k\geq0$, for the ease of exposition. We first observe from Algorithm~\ref{GT-SAGA} that~$\forall k\geq0,\forall i\in\mc{V}$,
\begin{align}\label{unbias}
\mbb{E}\big[\wh{\nabla}_i^k|\mc{F}^k\big] 
= \nabla f_i(\mb{x}_i^k) - \frac{1}{m}\sum_{j=1}^{m}\nabla f_{i,j}\big(\mb{z}_{i,j}^{k}\big).
\end{align}
In light of~\eqref{unbias}, we bound the variance of~$\mb{g}_i^k$ in the following:~$\forall k\geq0,\forall i\in\mc{V}$,
\begin{align}
&\mathbb{E}\big[\|\mb{g}_i^k-\nabla f_i(\mb{x}_i^k)\|^2|\mc{F}^k\big]\nonumber\\
=&~\mathbb{E}\big[\|\wh{\nabla}_i^k - \mathbb{E}\big[\wh{\nabla}_i^k|\mc{F}^k\big]\|^2|\mc{F}^k\big] \nonumber\\
\stackrel{(i)}{\leq}&~
\mathbb{E}\big[\|\wh{\nabla}_i^k\|^2\big|\mc{F}^k\big]\nonumber\\
=&~
\mathbb{E}\bigg[\sum_{j=1}^m\1_{\{\tau_i^k = j\}}\left\|\nabla f_{i,j}\big(\mb{x}_{i}^{k}\big) - \nabla f_{i,j}\big(\mb{z}_{i,j}^{k}\big)\right\|^2\big|\mc{F}^k\bigg]\nonumber\\
\stackrel{(ii)}{=}&~\frac{1}{m}\sum_{j=1}^m\left\|\nabla f_{i,j}\big(\mb{x}_{i}^{k}\big) - \nabla f_{i,j}\big(\mb{z}_{i,j}^k\big)\right\|^2 \nonumber\\
\stackrel{(iii)}{\leq}&~\frac{L^2}{m}\sum_{j=1}^m\left\|\mb{x}_i^k - \mb{z}_{i,j}^k\right\|^2 \nonumber\\
\leq&~2L^2\left\|\mb{x}_i^k - \ol{\mb{x}}^k\right\|^2 + 2L^2 t_i^k. \label{BVR_i}
\end{align}
where~$(i)$ the conditional variance decomposition,~$(ii)$ uses that $\|\nabla f_{i,j}(\mb{x}_{i}^{k}) - \nabla f_{i,j}(\mb{z}_{i,j}^{k})\|^2$ is~$\F^k$-measurable and that~$\tau_i^k$ is independent of~$\F^k$, and~$(iii)$ uses the~$L$-smoothness of each~$f_{i,j}$. Summing up~\eqref{BVR_i} over~$i$ from~$1$ to~$n$ gives~\eqref{BVR1}. Towards~\eqref{BVR2}, we have:~$\forall k\geq0$,
\begin{align}\label{BVR20}
\mathbb{E}\big[\|\ol{\mb g}^k\|^2&|\mc{F}^k\big] 
\stackrel{(i)}{=}\mathbb{E}\big[\|\ol{\mb g}^k - \ol{\nabla\mb{f}}(\mb{x}^k)\|^2 |\mc{F}^k\big]
+\|\ol{\nabla\mb{f}}(\mb{x}^k)\|^2  \nonumber\\
\stackrel{(ii)}{=}&~\frac{1}{n^2}\mathbb{E}\big[\|\mb{g}^k - \nabla\mb{f}(\mb{x}^k)\|^2 |\mc{F}^k\big]
+\|\ol{\nabla\mb{f}}(\mb{x}^k)\|^2, 
\end{align}
where~$(i)$ uses that~$\mathbb{E}[\ol{\mb g}^k |\mc{F}^k] = \ol{\nabla\mb{f}}(\mb{x}^k)$ and that~$\ol{\nabla\mb{f}}(\mb{x}^k)$ is~$\mc{F}^k$-measurable while~$(ii)$ uses that, whenever~$i\neq j$, 
$\mbb{E}[\langle\mb{g}_i^k-\nabla f_i(\mb{x}_i^k),\mb{g}_j^k-\nabla f_j(\mb{x}_j^k)\rangle|\mc{F}^k] = 0,$  since~$\tau_i^k$ is independent of~$\sigma(\sigma(\tau_j^k),\mc{F}^k)$ and~$\mathbb{E}[\mb{g}^k |\mc{F}^k] =\nabla\mb{f}(\mb{x}^k)$. The proof follows by applying~\eqref{BVR1} to~\eqref{BVR20}.
\end{proof}

\subsection{A descent inequality}\label{sec_ds}
In this subsection, we provide a key descent inequality that characterizes the expected decrease of the global function value at each iteration of \textbf{\texttt{GT-SAGA}}.
\begin{lemma}\label{descent}
If~$0<\alpha\leq\frac{1}{2L}$, then~$\forall k\geq0$,
\begin{align*}
\mathbb{E}\big[F(\ol{\mb x}^{k+1})|\F^k\big] 
\leq&~F(\ol{\mb x}^{k})
-\frac{\alpha}{2}\|\nabla F(\ol{\mb x}^{k})\|^2
-\frac{\alpha}{4}\|\ol{\nabla\mb{f}}(\mb{x}^{k})\|^2 \nonumber\\
&+ \frac{\alpha L^2}{n}\|\mb{x}^k - \J\mb{x}^k\|^2
+ \frac{\alpha^2L^3}{n}t^k.
\end{align*}
\end{lemma}

\begin{proof}
Since~$F$ is~$L$-smooth, we have~\cite{book_beck}:~$\forall\mb{x},\mb{y}\in\mathbb{R}^p$,
\begin{equation}\label{DL}
F(\mb y) \leq F(\mb x) + \langle \nabla F(\mb x), \mb{y} - \mb{x}\rangle
+ \frac{L}{2}\left\|\mb y - \mb x\right\|^2.
\end{equation}
We multiply~\eqref{gtsaga_x} by~$\frac{1}{n}(\mb{1}_n^\top\otimes\I_p)$ and use Lemma~\ref{basic}(\ref{track}) to obtain:
$\ol{\mb{x}}^{k+1} 
= \ol{\mb{x}}^{k} - \alpha\ol{\mb y}^{k+1}
= \ol{\mb{x}}^{k} - \alpha\ol{\mb g}^k,\forall k\geq0.$
Setting~$\mb{y} = \ol{\mb x}^{k+1}$ and~$\mb{x} = \ol{\mb x}^{k}$ in~\eqref{DL} obtains:~$\forall k\geq0$,
\begin{equation}\label{D0}
F(\ol{\mb x}^{k+1}) \leq F(\ol{\mb x}^{k}) -\alpha \langle \nabla F(\ol{\mb x}^{k}), \ol{\mb g}^k\rangle
+ \frac{\alpha^2 L}{2}\|\ol{\mb g}^k\|^2. 
\end{equation}
Conditioning~\eqref{D0} with respect to~$\mc{F}^k$, since~$\nabla F(\ol{\mb x}^k)$ is~$\mc{F}^k$-measurable, we have: 
\begin{align}\label{D1}
\mathbb{E}\left[F(\ol{\mb x}^{k+1}) | \mc{F}^k\right] \leq&~ F(\ol{\mb x}^{k})
-\alpha \left\langle \nabla F(\ol{\mb x}^{k}), \ol{\nabla\mb{f}}(\mb{x}^k)\right\rangle
\nonumber\\
&+ \frac{\alpha^2 L}{2}\mathbb{E}\left[\|\ol{\mb g}^k\|^2|\mc{F}^k\right].
\end{align}
Using~$2\langle \mb{a},\mb{b} \rangle = \|\mb a\|^2 + \|\mb b\|^2 
- \|\mb{a}-\mb{b}\|^2$,$\forall\mb{a},\mb{b}\in\mbb{R}^p$, in~\eqref{D1}, we obtain:~$\forall k\geq0$,
\begin{align}\label{D2}
\mathbb{E}&\big[F(\ol{\mb x}^{k+1})|\mc{F}^k\big] 
\leq F(\ol{\mb x}^{k})
-\frac{\alpha}{2}\|\nabla F(\ol{\mb x}^{k})\|^2 
-\frac{\alpha}{2}\|\ol{\nabla\mb{f}}(\mb{x}^{k})\|^2 \nonumber\\
&+ \frac{\alpha}{2} \|\nabla F(\ol{\mb x}^{k})-\ol{\nabla\mb{f}}(\mb{x}^k)\|^2 
+ \frac{\alpha^2 L}{2}\mathbb{E}\left[\|\ol{\mb g}^k\|^2|\mc{F}^k\right].
\end{align}
Applying Lemma~\ref{basic}(\ref{Lbound}) and~\eqref{BVR2} to~\eqref{D2}, we have:~$\forall k\geq0$,
\begin{align}\label{D3}
&\mathbb{E}[F(\ol{\mb x}^{k+1})|\mc{F}^k] \nonumber\\
\leq&~F(\ol{\mb x}^{k})
-\frac{\alpha}{2}\|\nabla F(\ol{\mb x}^{k})\|^2
-\frac{\alpha(1-\alpha L)}{2}\|\ol{\nabla\mb{f}}(\mb{x}^{k})\|^2 \nonumber\\
&+ \Big(\frac{\alpha L^2}{2n}+\frac{\alpha^2L^3}{n^2}\Big) \|\mb{x}^k - \J\mb{x}^k\|^2 
+ \frac{\alpha^2L^3}{n}t^k.
\end{align}
The proof follows by the fact that if~$0<\alpha\leq\frac{1}{2L}$, we have~$-\frac{\alpha\left(1-\alpha L\right)}{2}\leq-\frac{\alpha}{4}$ and~$\frac{\alpha L^2}{2n}+\frac{\alpha^2 L^3}{n^2}\leq\frac{\alpha L^2}{n}$.
\end{proof}
Compared with the corresponding descent inequality for centralized batch gradient descent~\cite{book_beck}, Lemma~\ref{descent} exhibits two additional bias terms, i.e.,~$\|\mb{x}^k-\J\mb{x}^k\|$ and~$t^k$, that are due to the decentralized nature of the problem and sampling. To establish the convergence of \textbf{\texttt{GT-SAGA}}, we therefore bound these bias terms by  $\|\ol{\nf}(\mb{x}^k)\|$ and show that they are dominated by the descent effect $-\|\ol{\nf}(\mb{x}^k)\|$.  

\subsection{Bounds on the auxiliary sequence~$t^k$}\label{sec_t}
In this subsection, we analyze the evolution of the auxiliary sequence~$t^k$ and establish useful bounds.
\begin{lemma}\label{aux} 
The following inequality holds:~$\forall k\geq0$,
\begin{align*}
\mbb{E}\left[t^{k+1}|\mc{F}^k\right]
\leq&~\theta t^k
+\Big(2\alpha^2+\frac{\alpha}{\beta}\Big)\|\ol{\nabla\mb{f}}(\mb{x}^k)\|^2 
\nonumber\\
&+\Big(\frac{2\alpha^2L^2}{n}+\frac{2}{m}\Big)\frac{1}{n}\|\mb{x}^{k}-\J\mb{x}^{k}\|^2,
\end{align*}
where the parameter~$\theta\in\mbb{R}$ is given by
\begin{align}\label{theta}
\theta := 1 - \frac{1}{m} + \alpha\beta + \frac{2\alpha^2L^2}{n},
\end{align}
and~$\beta>0$ is an arbitrary positive constant.
\end{lemma}

\begin{proof}
We define~$\mc{A}^k := \sigma\left(\cup_{i=1}^n\sigma(\tau_i^k),\mc{F}^k\right)$ and clearly $\F^k\subseteq\mc{A}^k$.
By the tower property of the conditional expectation, we have:~$\forall i\in\mc{V},\forall k\geq0$, 
\begin{align}\label{p0}
\mbb{E}\big[t_i^{k+1}|\mc{F}^k\big] 
=\frac{1}{m}\sum_{j=1}^m\E\big[\mathbb{E}\big[\|\ol{\mb{x}}^{k+1} - \mb{z}_{i,j}^{k+1}\|^2|\mc{A}^k\big]|\mc{F}^k\big].
\end{align}
Since~$s_i^k$ is independent of~$\mc{A}^k$ under Assumption~\ref{sample}, we have:~$\forall i\in\mc{V},\forall j\in\{1,\cdots,m\}, k\geq0$, 
\begin{align}\label{p_indep}
\mbb{E}\Big[\1_{\{s_i^k = j\}} | \mc{A}^k\Big] = \frac{1}{m} 
~\mbox{and}~ 
\mbb{E}\Big[\1_{\{s_i^k \neq j\}} | \mc{A}^k\Big] = 1 - \frac{1}{m}. 
\end{align}
In light of~\eqref{p_indep}, we have:~$\forall i\in\mc{V},\forall j\in\{1,\cdots,m\}, k\geq0$,
\begin{align}\label{p0_1}
&\mathbb{E}\left[\|\ol{\mb{x}}^{k+1} - \mb{z}_{i,j}^{k+1}\|^2|\mc{A}^k\right]
\nonumber\\
=&~\mathbb{E}\Big[\left\|\ol{\mb{x}}^{k+1} - \left(\1_{\{s_i^k = j\}}\mb{x}_i^k + \1_{\{s_i^k \neq j\}}\mb{z}_{i,j}^k\right)\right\|^2\Big| \mc{A}^k\Big] \nonumber\\
=&~\mathbb{E}\big[\|\ol{\mb{x}}^{k+1}\|^2|\mc{A}^k\big] + \E\Big[\left\|\1_{\{s_i^k = j\}}\mb{x}_i^k 
+ \1_{\{s_i^k \neq j\}}\mb{z}_{i,j}^k\right\|^2\Big|\mc{A}^k\Big]\nonumber\\
&-2\E\Big[\left\langle \ol{\mb{x}}^{k+1}, \1_{\{s_i^k = j\}}\mb{x}_i^k + \1_{\{s_i^k \neq j\}}\mb{z}_{i,j}^k\right\rangle\Big|\mc{A}^k\Big] \nonumber\\
\stackrel{(i)}{=}&\|\ol{\mb{x}}^{k+1}\|^2 -2\bigg\langle \ol{\mb{x}}^{k+1}, \frac{1}{m}\mb{x}_i^k + \Big(1-\frac{1}{m}\Big)\mb{z}_{i,j}^k\bigg\rangle  \nonumber\\
&+ \frac{1}{m}\|\mb{x}_i^k\|^2 + \Big(1-\frac{1}{m}\Big)\|\mb{z}_{i,j}^k\|^2, \nonumber\\
=&~\frac{1}{m}\|\ol{\mb{x}}^{k+1} - \mb{x}_i^k\|^2
+ \Big(1-\frac{1}{m}\Big)\|\ol{\mb{x}}^{k+1} - \mb{z}_{i,j}^k\|^2
\end{align}
where~$(i)$ uses~\eqref{p_indep} and that~$\ol{\mb x}^{k+1}, \mb{x}_i^k$, and~$\mb{z}_{i,j}^k$ are~$\mc{A}^k$-measurable. Using~\eqref{p0_1} in~\eqref{p0}, we obtain:~$\forall i\in\mc{V},\forall k\geq0$,
\begin{align}\label{p0_2}
\mbb{E}\big[t_i^{k+1}&|\mc{F}^k\big] 
= \frac{1}{m}\mbb{E}\left[\left\|\ol{\mb{x}}^{k+1} - \mb{x}_i^{k}\right\|^2\Big|\mc{F}^k\right] \nonumber\\
&+ \Big(1-\frac{1}{m}\Big)\frac{1}{m}\sum_{j=1}^m\mbb{E}\left[\left\|\ol{\mb{x}}^{k+1} - \mb{z}_{i,j}^{k}\right\|^2\Big|\mc{F}^k\right]. 
\end{align}
We next bound the two terms on the RHS of~\eqref{p0_2} separately.
For the first term, we have:~$\forall i\in\mc{V}$,~$k\geq0$,
\begin{align}\label{p1_1}
&\mbb{E}\big[\|\ol{\mb{x}}^{k+1} - \mb{x}_i^{k}\|^2|\mc{F}^k\big] \nonumber\\
=&~\mbb{E}\big[\|\ol{\mb{x}}^{k+1} - \ol{\mb{x}}^{k} + \ol{\mb{x}}^{k} - \mb{x}_i^k\|^2|\mc{F}^k\big]
 \nonumber\\
=&~\alpha^2\mbb{E}\big[\|\ol{\mb g}^k\|^2|\mc{F}^k\big]
-2\big\langle\alpha\ol{\nabla\mb{f}}(\mb{x}^k),\ol{\mb{x}}^{k} - \mb{x}_i^k\big\rangle
+\|\ol{\mb{x}}^{k} - \mb{x}_i^{k}\|^2 \nonumber\\
\leq&~\alpha^2\mbb{E}\big[\|\ol{\mb g}^k\|^2|\mc{F}^k\big]
+\alpha^2\|\ol{\nabla\mb{f}}(\mb{x}^k)\|^2
+2\|\mb{x}_i^{k}-\ol{\mb{x}}^{k}\|^2, 
\end{align}
where the last line uses the Cauchy-Schwarz inequality. 
Towards the second term on the RHS of~\eqref{p0_2}, we have:~$\forall i\in\mc{V}$,~$j\in\{1,\cdots,m\}$,~$\forall k\geq0$,~$\forall \beta>0$,
\begin{align}\label{p1_2}
&\mbb{E}\big[\|\ol{\mb{x}}^{k+1} - \mb{z}_{i,j}^{k}\|^2|\mc{F}^k\big]
\nonumber\\
=&~\mbb{E}\big[\|\ol{\mb{x}}^{k+1}-\ol{\mb{x}}^{k}+\ol{\mb{x}}^{k} - \mb{z}_{i,j}^{k}\|^2\big|\mc{F}^k\big] \nonumber\\
=&~\alpha^2\mbb{E}\big[\|\ol{\mb g}^k\|^2|\mc{F}^k\big]
- 2\alpha\big\langle\ol{\nabla\mb{f}}(\mb{x}^k),\ol{\mb{x}}^{k} - \mb{z}_{i,j}^{k}\big\rangle
+ \|\ol{\mb{x}}^{k} - \mb{z}_{i,j}^{k}\|^2 \nonumber\\
\leq&~\alpha^2\mbb{E}\big[\|\ol{\mb g}^k\|^2|\mc{F}^k\big]
+ (1+\alpha\beta)\|\ol{\mb{x}}^{k} - \mb{z}_{i,j}^{k}\|^2 + \frac{\alpha}{\beta}\|\ol{\nabla\mb{f}}(\mb{x}^k)\|^2,
\end{align}
where the last line uses Young's inequality. 
Now, we apply~\eqref{p1_1} and~\eqref{p1_2} to~\eqref{p0_2} to obtain:~$\forall i\in\mc{V}$,~$\forall k\geq0$,
\begin{align}\label{p2}
&\mbb{E}\big[t_i^{k+1}|\mc{F}^k\big]
\leq\Big(1-\frac{1}{m}\Big)(1+\alpha\beta)t_i^k+\alpha^2\mbb{E}\big[\|\ol{\mb{g}}^k\|^2|\mc{F}^k\big] \nonumber\\
&+\frac{2}{m}\|\mb{x}_i^{k}-\ol{\mb{x}}^{k}\|^2+\Big(\frac{\alpha^2}{m}+\Big(1-\frac{1}{m}\Big)\frac{\alpha}{\beta}\Big)\|\ol{\nabla\mb{f}}(\mb{x}^k)\|^2. 
\end{align}
We average~\eqref{p2} over~$i$ from~$1$ to~$n$ and use~\eqref{BVR2} in the resulting inequality to obtain:~$\forall k\geq0$,
\begin{align}\label{p3}
\mbb{E}\big[t^{k+1}|\mc{F}^k\big] 
\leq&~\Big(\frac{2\alpha^2L^2}{n}+\frac{2}{m}\Big)\frac{1}{n}\|\mb{x}^{k}-\J\mb{x}^{k}\|^2 \nonumber\\
+&\left(\alpha^2+\frac{\alpha^2}{m}+\Big(1-\frac{1}{m}\Big)\frac{\alpha}{\beta}\right)\left\|\ol{\nabla\mb{f}}(\mb{x}^k)\right\|^2 \nonumber\\
+&\left(\frac{2\alpha^2L^2}{n}+\Big(1-\frac{1}{m}\Big)(1+\alpha\beta)\right)t^k.
\end{align}
We conclude by using~$\frac{1}{m}+1\leq2$ and~$1-\frac{1}{m}\leq1$ in~\eqref{p3}.
\end{proof}
Next, we specify some particular choices of~$\beta$ and the range of~$\alpha$ in Lemma~\ref{aux} to obtain useful bounds on the auxiliary sequence~$t^k$. The following corollary shows that~$t^k$ has an intrinsic contraction property. 
\begin{corollary}\label{aux_ctr}
If~$0<\alpha\leq\frac{\sqrt{n}}{\sqrt{8m}L}$, then~$\forall k\geq0$,
\begin{align*}
\mbb{E}\left[t^{k+1}|\mc{F}^k\right]
\leq&~
\Big(1-\frac{1}{4m}\Big)t^k
+4m\alpha^2\|\ol{\nabla\mb{f}}(\mb{x}^k)\|^2 \nonumber\\
&+\frac{9}{4mn}\|\mb{x}^{k}-\J\mb{x}^{k}\|^2.
\end{align*}
\end{corollary}
\begin{proof}
We choose~$\beta = \frac{1}{2m\alpha}$ in Lemma~\ref{aux} to obtain: if~$0<\alpha\leq\frac{\sqrt{n}}{\sqrt{8m}L}$, i.e.,~$\frac{2\alpha^2 L^2}{n}\leq\frac{1}{4m}$, then
\begin{align}
&\theta = 1 - \frac{1}{m} + \alpha\beta + \frac{2\alpha^2L^2}{n}
\leq 1 - \frac{1}{4m}. \label{aux_c1}\\
&2\alpha^2 + \frac{\alpha}{\beta} = 2\alpha^2 + 2m\alpha^2\leq4m\alpha^2. \label{aux_c2} \\
&\frac{2\alpha^2L^2}{n} + \frac{2}{m} \leq \frac{1}{4m} + \frac{2}{m} = \frac{9}{4m}.
\label{aux_c3}
\end{align}
We conclude by applying~\eqref{aux_c1},~\eqref{aux_c2},~\eqref{aux_c3} to Lemma~\ref{aux}.
\end{proof}

The following corollary of Lemma~\ref{aux} will be only used to bound~$\mathbb{E}[\|\mb{g}^{k+1}-\nabla\mb{f}(\mb{x}^{k+1})\|^2|\mc{F}^k]$. 
\begin{corollary}\label{aux_1}
If~$0<\alpha\leq\frac{\sqrt{n}}{\sqrt{8m}L}$, then~$\forall k\geq0$,
\begin{align*}
\mbb{E}\big[t^{k+1}|\mc{F}^k\big]
\leq&~
2t^k
+3\alpha^2\|\ol{\nabla\mb{f}}(\mb{x}^k)\|^2 
+\frac{9}{4mn}\|\mb{x}^{k}-\J\mb{x}^{k}\|^2.
\end{align*}
\end{corollary}
\begin{proof}
Setting~$\beta = 1/\alpha$ in Lemma~\ref{aux}, we have: if~$0<\alpha\leq\frac{\sqrt{n}}{\sqrt{8m}L}$, i.e.,~$\frac{2\a^2L^2}{n}\leq\frac{1}{4m}$, then
\begin{align}
&\theta = 1 - \frac{1}{m} +\a\beta + \frac{2L^2\alpha^2}{n}
\leq2  \label{aux_b1} \\
&2\a^2 + \frac{\a}{\beta} = 3\a^2 \label{aux_b2} \\
&\frac{2\alpha^2 L^2}{n} + \frac{2}{m} \leq\frac{1}{4m} + \frac{2}{m}  = \frac{9}{4m} ,\label{aux_b3}
\end{align}
We conclude by applying~\eqref{aux_b1},~\eqref{aux_b2},~\eqref{aux_b3} to Lemma~\ref{aux}.
\end{proof}

With the help of~\eqref{consensus2},~\eqref{BVR1} and Corollary~\ref{aux_1}, we provide an upper bound on $\mathbb{E}[\|\mb{g}^{k+1}-\nabla\mb{f}(\mb{x}^{k+1})\|^2|\mc{F}^k]$.
\begin{lemma}\label{BVR3}
If~$0<\alpha\leq\frac{\sqrt{n}}{\sqrt{8m}L}$, then~$\forall k\geq0$,
\begin{align*}
\mathbb{E}&[\|\mb{g}^{k+1}-\nabla\mb{f}(\mb{x}^{k+1})\|^2|\mc{F}^k]
\leq
8.5L^2\|\mb{x}^{k} - \J\mb{x}^{k}\|^2 \!+ 4nL^2 t^k
 \nonumber\\
&+ 6n\alpha^2 L^2\|\ol{\nabla\mb{f}}(\mb{x}^k)\|^2
+ 4\alpha^2 L^2\mbb{E}\big[\|\mb{y}^{k+1}-\J\mb{y}^{k+1}\|^2|\mc{F}^k\big]. 
\end{align*}
\end{lemma}
\begin{proof}
By the tower property of the conditional expectation, we have:~$\forall k\geq0$,
\begin{align*}
&\mathbb{E}\left[\|\mb{g}^{k+1} - \nabla\mb{f}(\mb{x}^{k+1})\|^2|\mc{F}^k\right]
\nonumber\\
=&~\mbb{E}\left[\mathbb{E}\left[\|\mb{g}^{k+1}  -\nabla\mb{f}(\mb{x}^{k+1})\|^2|\mc{F}^{k+1}\right]|\mc{F}^{k}\right] \nonumber\\
\leq&~2L^2\mbb{E}\left[\|\mb{x}^{k+1} - \J\mb{x}^{k+1}\|^2|\mc{F}^k\right]
+ 2nL^2\mbb{E}\left[t^{k+1}|\mc{F}^k\right] \nonumber\\
\leq&~2L^2\left(2\|\mb{x}^{k}-\J\mb{x}^{k}\|^2 + 2\alpha^2\mbb{E}\left[\|\mb{y}^{k+1}-\J\mb{y}^{k+1}\|^2|\mc{F}^k\right]\right)  \nonumber\\
&+ 2nL^2\Big(2t^k
+3\alpha^2\|\ol{\nabla\mb{f}}(\mb{x}^k)\|^2 
+\frac{9}{4mn}\|\mb{x}^{k}-\J\mb{x}^{k}\|^2\Big),
\end{align*}
where the second line uses~\eqref{BVR1} and the third line uses~\eqref{consensus2} and Corollary~\ref{aux_1}. The desired inequality then follows.
\end{proof}

\subsection{Bounds on stochastic gradient tracking process}\label{sec_sgt}
In this subsection, we analyze the variance-reduced stochastic gradient tracking process~\eqref{gtsaga_y}. 

\begin{lemma}\label{gt0}
The following inequality holds:~$\forall k\geq0$,
\begin{align*}
&\E\big[\|\mb{y}^{k+2}-\J\mb{y}^{k+2}\|^2\big] \nonumber\\
\leq&~\lambda^2\E\big[\|\mb{y}^{k+1}-\J\mb{y}^{k+1}\|^2\big] 
+\lambda^2\E\big[\|\mb{g}^{k+1} - \mb{g}^{k}\|^2\big]
\nonumber\\
+&~2\E\big[\left\langle(\W-\J)\mb{y}^{k+1},\left(\W-\J\right)\left(\nf(\x^{k+1}) - \nf(\x^{k})\right)\right\rangle\big] \nonumber\\
+&~2\E\big[\left\langle(\W-\J)\mb{y}^{k+1},\left(\W-\J\right)\left(\nf(\x^{k}) - \mb{g}^{k}\right)\right\rangle\big].
\end{align*}
\end{lemma}
\begin{proof}
Using~\eqref{gtsaga_y} and~$\J\W = \J$, we have:~$\forall k\geq0$,
\begin{align}\label{y0}
&\|\mb{y}^{k+2}-\J\mb{y}^{k+2}\|^2 \nonumber\\
=&~\|\W\mb{y}^{k+1}-\J\mb{y}^{k+1}+\left(\W-\J\right)\left(\mb{g}^{k+1} - \mb{g}^{k}\right)\|^2 \nonumber\\
=&~\|\W\mb{y}^{k+1}-\J\mb{y}^{k+1}\|^2 +\|\left(\W-\J\right)\left(\mb{g}^{k+1} - \mb{g}^{k}\right)\|^2 \nonumber\\
&+2\left\langle\W\mb{y}^{k+1}-\J\mb{y}^{k+1},\left(\W-\J\right)\left(\mb{g}^{k+1} - \mb{g}^{k}\right)\right\rangle   \nonumber\\
\leq&~\lambda^2\|\mb{y}^{k+1}-\J\mb{y}^{k+1}\|^2 +\lambda^2\|\mb{g}^{k+1} - \mb{g}^{k}\|^2
\nonumber\\
&+2\left\langle\W\mb{y}^{k+1}-\J\mb{y}^{k+1},\left(\W-\J\right)\left(\mb{g}^{k+1} - \mb{g}^{k}\right)\right\rangle,
\end{align}
where the last line uses Lemma~\ref{basic}(\ref{W}) and $\left\|\W-\J\right\|= \lambda$. To proceed, we observe that~$\forall k\geq0$,
\begin{align}\label{y0_1}
&\E\left[\left\langle\W\mb{y}^{k+1}-\J\mb{y}^{k+1},\left(\W-\J\right)\left(\mb{g}^{k+1} - \mb{g}^{k}\right)\right\rangle|\F^{k+1}\right] \nonumber\\
=&~\left\langle\W\mb{y}^{k+1}-\J\mb{y}^{k+1},\left(\W-\J\right)\left(\nf(\x^{k+1}) - \mb{g}^{k}\right)\right\rangle \nonumber\\
=&~\left\langle\W\mb{y}^{k+1}-\J\mb{y}^{k+1},\left(\W-\J\right)\left(\nf(\x^{k+1}) - \nf(\x^{k})\right)\right\rangle \nonumber\\
&+\left\langle\W\mb{y}^{k+1}-\J\mb{y}^{k+1},\left(\W-\J\right)\left(\nf(\x^{k}) - \mb{g}^{k}\right)\right\rangle,
\end{align}
where the first line uses that~$\E[\g^{k+1}|\F^{k+1}] = \nf(\x^{k+1})$ and that~$\y^{k+1}$ and~$\g^{k}$ are~$\F^{k+1}$-measurable for all~$k\geq0$. We conclude by using~\eqref{y0_1} in~\eqref{y0} and taking the expectation.
\end{proof}

We next bound the third term in Lemma~\ref{gt0}.
\begin{lemma}\label{gt1}
The following inequality holds:~$\forall k\geq0$,
\begin{align*}
&\left\langle\W\mb{y}^{k+1}-\J\mb{y}^{k+1},\left(\W-\J\right)\left(\nf(\x^{k+1}) - \nf(\x^{k})\right)\right\rangle    \nonumber\\
\leq&~\left(\lambda\a L + 0.5\eta_1  + \eta_2\right)\lambda^2\|\mb{y}^{k+1}-\J\mb{y}^{k+1}\|^2 \nonumber\\
&+ 0.5\eta_1^{-1}\lambda^2\a^2L^2n\|\ol{\g}^{k}\|^2
+ \eta_2^{-1}\lambda^2L^2\|\x^{k}-\J\x^{k}\|^2,
\end{align*}
where~$\eta_1>0$ and~$\eta_2>0$ are arbitrary.
\end{lemma}
\begin{proof}
Using Lemma~\ref{basic}(\ref{W}) and $\|\W-\J\| = \lambda$, we have
\begin{align}\label{y1}
&\left\langle\W\mb{y}^{k+1}-\J\mb{y}^{k+1},\left(\W-\J\right)\left(\nf(\x^{k+1}) - \nf(\x^{k})\right)\right\rangle \nonumber\\
\leq&~\lambda^2L\|\mb{y}^{k+1}-\J\mb{y}^{k+1}\|\|\x^{k+1}-\x^{k}\|, \qquad\forall k\geq0.
\end{align}
Observe that~$\forall k\geq0$,
\begin{align}\label{y11}
&\|\x^{k+1}-\x^{k}\| \nonumber\\
=&~\|\x^{k+1} - \J\x^{k+1} + \J\x^{k+1} -\J\x^{k} + \J\x^{k} -\x^{k}\|
\nonumber\\
\leq&~\|\x^{k+1} - \J\x^{k+1}\| + \sqrt{n}\a\|\ol{\g}^{k}\| + \|\x^{k}-\J\x^{k}\| \nonumber\\
\leq&~2\|\x^{k}-\J\x^{k}\| + \sqrt{n}\a\|\ol{\g}^{k}\| 
+ \a\lambda\|\y^{k+1}-\J\y^{k+1}\|,
\end{align}
where the last line is due to~\eqref{consensus3}. We use~\eqref{y11} in~\eqref{y1} to obtain
\begin{align}\label{y12}
&\left\langle\W\mb{y}^{k+1}-\J\mb{y}^{k+1},\left(\W-\J\right)\left(\nf(\x^{k+1}) - \nf(\x^{k})\right)\right\rangle \nonumber\\
\leq&~\lambda^3\a L\|\mb{y}^{k+1}-\J\mb{y}^{k+1}\|^2 
+ \lambda^2\|\mb{y}^{k+1}-\J\mb{y}^{k+1}\|\sqrt{n}\a L\|\ol{\g}^{k}\|
\nonumber\\
&+ 2\lambda^2\|\mb{y}^{k+1}-\J\mb{y}^{k+1}\|L\|\mb{x}^{k}-\J\mb{x}^{k}\|, \qquad\forall k\geq0.
\end{align}
By Young's inequality, we have:~$\forall k\geq0$, for some~$\eta_1>0$,
\begin{align}\label{y13}
&\lambda^2\|\mb{y}^{k+1}-\J\mb{y}^{k+1}\|\sqrt{n}\a L\|\ol{\g}^{k}\|
\nonumber\\
\leq&~0.5\lambda^2\left(\eta_1\|\mb{y}^{k+1}-\J\mb{y}^{k+1}\|^2 
+ \eta_1^{-1}n\a^2L^2\|\ol{\g}^{k}\|^2\right),
\end{align}
and,~$\forall k\geq0$, for some~$\eta_2>0$, 
\begin{align}\label{y14}
&2\lambda^2\|\mb{y}^{k+1}-\J\mb{y}^{k+1}\|L\|\mb{x}^{k}-\J\mb{x}^{k}\|\nonumber\\
\leq&~\lambda^2\eta_2\|\mb{y}^{k+1}-\J\mb{y}^{k+1}\|^2
+ \lambda^2\eta_2^{-1}L^2\|\mb{x}^{k}-\J\mb{x}^{k}\|^2.
\end{align}
The proof follows by applying~\eqref{y13} and~\eqref{y14} to~\eqref{y12}.
\end{proof}

We next bound the fourth term in Lemma~\ref{gt0}.
\begin{lemma}\label{gt2}
The following inequality holds:~$\forall k\geq0$,
\begin{align*}
&\E\big[\left\langle\W\mb{y}^{k+1}-\J\mb{y}^{k+1},\left(\W-\J\right)\left(\nf(\x^{k}) - \mb{g}^{k}\right)\right\rangle\big]    \nonumber\\
\leq&~\E\big[\|\mb{g}^{k} -\nf(\x^{k})\|^2\big]/n.
\end{align*}
\end{lemma}
\begin{proof}
In the following, we denote~$\nf^k:=\nf(\x^{k})$ to~simplify the notation.
Observe that~$\forall k\geq0$,
\begin{align}\label{y3}
&\E\left[\left\langle\W\mb{y}^{k+1}-\J\mb{y}^{k+1},\left(\W-\J\right)\left(\nf^{k} - \mb{g}^{k}\right)\right\rangle | \F^{k}\right]   \nonumber\\
\stackrel{(i)}{=}&~\E\left[\left\langle\W^2(\y^{k} + \g^{k} - \g^{k-1}),(\W-\J)(\nf^{k} - \mb{g}^{k})\right\rangle|\F^{k}\right] \nonumber\\
\stackrel{(ii)}{=}&~\E\left[\left\langle\W^2\g^{k},\left(\W-\J\right)\left(\nf^{k} - \mb{g}^{k}\right)\right\rangle|\F^{k}\right] \nonumber\\
\stackrel{(iii)}{=}&~\E\left[\left\langle\W^2\left(\g^{k} - \nf^k\right),\left(\W-\J\right)\left(\nf^{k} - \mb{g}^{k}\right)\right\rangle|\F^{k}\right] \nonumber\\
\stackrel{(iv)}{=}&~\E\left[(\mb{g}^{k}-\nf^{k})^\top(\J-\W^\top\W^2)(\mb{g}^{k}-\nf^{k})|\F^{k}\right],
\end{align}
{\color{black}where~$(i)$ uses~\eqref{gtsaga_y} and~$\J\W = \J$,~$(ii)$ and~$(iii)$ use that $\y^{k}$, $\g^{k-1}$ and $\nf^k$ are $\F^{k}$-measurable and that~$\E[\g^{k}|\F^{k}] = \nf^{k}$ for all~$k\geq0$, and~$(iv)$ uses $\J\W = \J$.} Since, whenever $i\neq j\in\mc{V}$, $\mbb{E}\left[\big\langle\mb{g}_i^{k}-\nabla f_i(\mb{x}_i^{k}),\mb{g}_j^{k}-\nabla f_j(\mb{x}_j^{k})\big\rangle|\mc{F}^{k}\right] = 0,$
and~$\W^\top\W^2$ is nonnegative, we have:~$\forall k\geq0$,
\begin{align}\label{y31}
&\E\left[(\mb{g}^{k}-\nf^{k})^\top\left(\J-\W^\top\W^2\right)\left(\mb{g}^{k}-\nf^{k}\right)|\F^{k}\right]    \nonumber\\
=&~\E\left[(\mb{g}^{k}-\nf^{k})^\top\mbox{diag}\left(\J-\W^\top\W^2\right)\left(\mb{g}^{k}-\nf^{k}\right)|\F^{k}\right] \nonumber\\
\leq&~\E\left[(\mb{g}^{k}-\nf^{k})^\top\mbox{diag}(\J)\left(\mb{g}^{k}-\nf^{k}\right)|\F^{k}\right].
\end{align}
The proof follows by taking the expectation of~\eqref{y31}.
\end{proof}

We finally bound the second term in Lemma~\ref{gt0}.
\begin{lemma}\label{gt3}
The following inequality holds:~$\forall k\geq0$,
\begin{align*}
\E\big[\|\g^{k+1} -& \g^{k}\|^2\big] 
\leq 12\lambda^2\a^2L^2\E\big[\|\y^{k+1}-\J\y^{k+1}\|^2\big] \nonumber\\
+& 2\E\big[\|\g^{k}-\nf(\x^{k})\|^2\big]
+ \E\big[\|\g^{k+1} -\nf(\x^{k+1})\|^2\big] \nonumber\\
+& 18L^2\E\big[\|\x^{k}-\J\x^{k}\|^2\big]
+ 6n\a^2L^2\E\big[\|\ol{\g}^{k}\|^2\big] \nonumber.
\end{align*}
\end{lemma}

\begin{proof}
Since~$\g^{k}$ and~$\nf(\x^{k+1})$ are~$\F^{k+1}$-measurable, and $\E[\g^{k+1}|\F^{k+1}] = \nf(\x^{k+1})$, we have:~$\forall k\geq0$,
\begin{align}\label{y4_0}
&\E\left[\|\g^{k+1} - \g^{k}\|^2|\F^{k+1}\right] \nonumber\\
=&~\E\left[\|\g^{k+1} -\nf(\x^{k+1})\|^2|\F^{k+1}\right] + 
\|\nf(\x^{k+1}) - \g^{k}\|^2 \nonumber\\
\leq&~\E\left[\|\g^{k+1} -\nf(\x^{k+1})\|^2|\F^{k+1}\right] \nonumber\\
&+ 2\|\nf(\x^{k+1}) - \nf(\x^{k})\|^2 
+ 2\|\nf(\x^{k}) - \g^{k}\|^2 \nonumber\\
\leq&~\E\left[\|\g^{k+1} -\nf(\x^{k+1})\|^2|\F^{k+1}\right] \nonumber\\
&+ 2L^2\|\x^{k+1} - \x^{k}\|^2 
+ 2\|\nf(\x^{k}) - \g^{k}\|^2.
\end{align}
Similar to the derivation of~\eqref{y11}, we have:~$\forall k\geq0$,
\begin{align*}
&\|\x^{k+1}-\x^{k}\|^2  \nonumber\\
\leq&~3\|\x^{k+1} - \J\x^{k+1}\|^2 + 3n\a^2\|\ol{\g}^{k}\|^2 
+ 3\|\x^{k}-\J\x^{k}\|^2 \nonumber\\
\leq&~9\|\x^{k}-\J\x^{k}\|^2 
+ 3n\a^2\|\ol{\g}^{k}\|^2 + 6\a^2\lambda^2\|\y^{k+1}-\J\y^{k+1}\|^2,
\end{align*}
where the last line is due to~\eqref{consensus2}. We conclude by applying the last line above to~\eqref{y4_0} and taking the expectation.
\end{proof}

Now, we apply Lemma~\ref{gt1},~\ref{gt2},~\ref{gt3} to Lemma~\ref{gt0}.
\begin{lemma}\label{GT_general}
The following inequality holds:~$\forall k\geq0$,
\begin{align*}
&\E\left[\|\mb{y}^{k+2}-\J\mb{y}^{k+2}\|^2\right] \nonumber\\
\leq&~\left(1+ 2\lambda\a L + \eta_1  + 2\eta_2 + 12\lambda^2\a^2L^2\right)\lambda^2\nonumber\\
&\qquad\qquad\qquad\qquad\qquad\times\E\left[\|\mb{y}^{k+1}-\J\mb{y}^{k+1}\|^2\right] \nonumber\\
&+ \left(2\eta_2^{-1}+18\right)\lambda^2L^2\E\left[\|\x^{k}-\J\x^{k}\|^2\right] \nonumber\\
&+ \left(\eta_1^{-1} + 6\right)\lambda^2\a^2L^2n\E\left[\|\ol{\g}^{k}\|^2\right]
 \nonumber\\
&+ \left(2\lambda^2+2/n\right)\E\left[\|\mb{g}^{k} -\nf(\x^{k})\|^2\right] \nonumber\\
&+ \lambda^2\E\left[\|\mb{g}^{k+1} -\nf(\x^{k+1})\|^2\right].
\end{align*}
\end{lemma}
\begin{proof}
Apply Lemma~\ref{gt1},~\ref{gt2},~\ref{gt3} to Lemma~\ref{gt0}.
\end{proof}

Finally, we use Lemma~\ref{BVR} and~\ref{BVR3} to refine Lemma~\ref{GT_general} and establish a contraction in the gradient tracking process.

\begin{lemma}\label{GT_final}
If~$0<\a\leq\min\Big\{\frac{1-\lambda^2}{16\lambda},\frac{\sqrt{n}}{\sqrt{8m}}\Big\}\frac{1}{L}$, then we have:~$\forall k\geq0$,
\begin{align*}
&\E\big[\|\mb{y}^{k+2}-\J\mb{y}^{k+2}\|^2\big] \nonumber\\
\leq&~\frac{1+\lambda^2}{2}\E\big[\|\mb{y}^{k+1}-\J\mb{y}^{k+1}\|^2\big]
+ \frac{30.5L^2}{1-\lambda^2}\E\big[\|\x^{k}-\J\x^{k}\|^2\big] \\
&+ \frac{97L^2n}{8}\E\big[t^k\big] 
+ \frac{16\lambda^2\a^2L^2n}{1-\lambda^2}\E\big[\|\ol{\nabla\mb{f}}(\mb{x}^k)\|^2\big].
\end{align*}
\end{lemma}
\begin{proof}
We apply Lemma~\ref{BVR} and~\ref{BVR3} to Lemma~\ref{GT_general} to obtain: if~$0<\a\leq\frac{\sqrt{n}}{\sqrt{8m}L}$, then~$\forall k\geq0$,
\begin{align}\label{GT_BVR}
&\E\left[\|\mb{y}^{k+2}-\J\mb{y}^{k+2}\|^2\right] \nonumber\\
\leq&\left(1+ 2\lambda\a L + \eta_1  + 2\eta_2 + \left(12\lambda^2+4\right)\a^2L^2\right)\lambda^2 \nonumber\\
&\qquad\times\E\left[\|\mb{y}^{k+1}-\J\mb{y}^{k+1}\|^2\right] \nonumber\\
+&\bigg((2\eta_2^{-1}+18)\lambda^2+(\eta_1^{-1}+6)\frac{2\lambda^2\a^2L^2}{n}+ \frac{4}{n} + 12.5\lambda^2\bigg) \nonumber\\
&\qquad\times L^2\E\left[\|\x^{k}-\J\x^{k}\|^2\right] \nonumber\\
+&\Big(2(\eta_1^{-1} + 6)\lambda^2\a^2L^2+\left(2\lambda^2+1/n\right)4n\Big)L^2\E\big[t^k\big]
\nonumber\\
+& (\eta_1^{-1}+12)\lambda^2\alpha^2 L^2n \E\left[\|\ol{\nabla\mb{f}}(\mb{x}^k)\|^2\right]. 
\end{align}
We fix~$\eta_1 = \frac{1-\lambda^2}{16\lambda^2}$ and~$\eta_2 = \frac{1-\lambda^2}{8\lambda^2}$. It can then be verified that 
$
1+ 2\lambda\a L + \eta_1  + 2\eta_2 + (12\lambda^2+4)\a^2L^2 
\leq \frac{1+\lambda^2}{2\lambda^2},  
$
if~$0<\a\leq\frac{1-\lambda^2}{16\lambda L}$. 
The proof then follows by applying this inequality and the values of~$\eta_1$ and~$\eta_2$ to~\eqref{GT_BVR}. 
\end{proof}

\subsection{Proof of Theorem~\ref{main_ncvx}}\label{sec_Th1_proof}
In this subsection, we prove the convergence of \textbf{\texttt{GT-SAGA}}
for general smooth non-convex functions.
To this aim, we write the contraction inequalities in~\eqref{consensus1}, Corollary~\ref{aux_ctr}, and Lemma~\ref{GT_final} as a linear time-invariant (LTI) dynamics that jointly characterizes the evolution of the consensus, gradient tracking, and the auxiliary sequence~$t^k$.
\begin{prop}\label{LTI_ncvx}
If~$0<\a\leq\min\left\{\frac{1-\lambda^2}{16\lambda},\frac{\sqrt{n}}{\sqrt{8m}}\right\}\frac{1}{L}$, then  
$$\mb{u}^{k+1} \leq \G_{\a}\mb{u}^{k} + \mb{b}^k, \qquad\forall k\geq0,$$
where~$\mb{u}^{k}\in\mbb{R}^3$,~$\G_{\a}\in\mbb{R}^{3\times3}$, and~$\mb{b}^k\in\mbb{R}^3$ are given by
\begin{align*}
\mb{u}^k :=&
\begin{bmatrix}
{\color{black}\mbb{E}\left[\frac{1}{n}\|\mb{x}^{k}-\J\mb{x}^{k}\|^2\right]}\\
\mathbb{E}\left[t^{k}\right]\\
\mbb{E}\left[\frac{1}{nL^2}\|\mb{y}^{k+1}-\J\mb{y}^{k+1}\|^2\right]
\end{bmatrix},
\quad
\mb{b} := 
\begin{bmatrix}
0\\
4m\alpha^2 \\
\frac{16\lambda^2\alpha^2}{1-\lambda^2}
\end{bmatrix}, \\
\G_\a :=&
\begin{bmatrix}
\frac{1+\lambda^2}{2} & 0 & \frac{2\lambda^2\alpha^2L^2}{1-\lambda^2} \\
\frac{9}{4m}&  1-\frac{1}{4m} & 0\\
\frac{30.5}{1-\lambda^2} & \frac{97}{8} & \frac{1+\lambda^2}{2}
\end{bmatrix},
\end{align*}
and~$\mb{b}^k := \mb{b}\mathbb{E}\left[\|\ol{\nabla\mb{f}}(\mb{x}^k)\|^2\right]$.
\end{prop}
We first derive the range of the step-size~$\a$ under which the spectral radius of~$\G_\a$ defined in Proposition~\ref{LTI_ncvx} is less than~$1$, with the help of the following Lemma from~\cite{matrix_analysis}.
\begin{lemma}\label{rho_bound}
Let~$\mb{X}\in\mathbb{R}^{d\times d}$ be a non-negative matrix and~$\mb{x}\in\mathbb{R}^d$ be a positive vector. If~$\mb{X}\mb{x}<\mb{x}$, then~$\rho(\mb{X})<1$. Moreover, if~$\mb{X}\mb{x}\leq\beta\mb{x}$, for some~$\beta\in\mathbb{R}$, then~$\rho(\mb{X})\leq\beta$.
\end{lemma}

\begin{lemma}
If~$0<\alpha\leq\min\left\{\frac{(1-\lambda^2)^2}{35\lambda},\frac{\sqrt{n}}{\sqrt{8m}}\right\}\frac{1}{L}$, then we have~$\rho(\G_\alpha)<1$ and thus~$\sum_{k=0}^{\infty}\G_\alpha^k = \left(\I_3 - \G_\alpha\right)^{-1}$.
\end{lemma}
\begin{proof}
In light of Lemma~\ref{rho_bound}, we find a positive vector~$\bds{\epsilon}=[\epsilon_1,\epsilon_2,\epsilon_3]^\top$ and the range of~$\alpha$ s.t. $\G_\alpha\bds{\epsilon} < \bds{\epsilon}$, i.e.,
\begin{align}
&\alpha^2 < \frac{(1-\lambda^2)^2}{4\lambda^2L^2}\dfrac{\epsilon_1}{\epsilon_3}, \label{e1}\\
&9\epsilon_1 < \epsilon_2, \label{e2}\\
&\frac{61}{\left(1-\lambda^2\right)^2}\epsilon_1 + \frac{97}{4\left(1-\lambda^2\right)}\epsilon_2 
<\epsilon_3,                    \label{e3}
\end{align}
Based on~\eqref{e2}, we set~$\epsilon_1 = 1$ and~$\epsilon_2 = 10$. Then based on~\eqref{e3}, we set~$\epsilon_3 = \frac{303.5}{(1-\lambda^2)^2}$. The proof follows by using the values of~$\epsilon_1$ and~$\epsilon_3$ in~\eqref{e1}.
\end{proof} 

Based on the LTI dynamics in Proposition~\ref{LTI_ncvx}, we derive the following lemma that is the key to establish the convergence of~$\textbf{\texttt{GT-SAGA}}$ for general smooth nonconvex functions. 
\begin{lemma}\label{incremental_sums_vec}
If~$0<\alpha\leq\min\Big\{\frac{(1-\lambda^2)^2}{35\lambda},\frac{\sqrt{n}}{\sqrt{8m}}\Big\}\frac{1}{L}$, then we have:~$\forall K\geq1$,
\begin{align*}
\sum_{k=0}^{K}\mb{u}^k
\leq\left(\I-\G_\a\right)^{-1}\bigg(\mb{u}^0 
+ \mb{b}\sum_{k=0}^{K-1}\mbb{E}\big[\|\ol{\nabla \mb{f}}(\mb{x}^k)\|^2\big]\bigg).
\end{align*}
\end{lemma}
\begin{proof}
We recursively apply the dynamics in~Proposition~\ref{LTI_ncvx} to obtain:
$\mb{u}^{k} 
\leq \G_\a^k\mb{u}^0 + \sum_{r=0}^{k-1}\G_\a^r\mb{b}^{k-1-r},
\forall k\geq1.
$
We sum this inequality over~$k$ to obtain:~$\forall K\geq1$,
\begin{align*}
\sum_{k=0}^{K}\mb{u}^k
\leq&~\sum_{k=0}^{K}\G_\alpha^k\mb{u}^0+\sum_{k=1}^{K}\sum_{r=0}^{k-1}\G_\alpha^r\mb{b}^{k-1-r} \nonumber\\
\leq&~\left(\sum_{k=0}^{\infty}\G_\alpha^k\right)\mb{u}^0+\sum_{k=0}^{K-1}\left(\sum_{k=0}^{\infty}\G_\a^k\right)\mb{b}^k.
\end{align*}
The proof follows by~$\sum_{k=0}^{\infty}\G_\alpha^k = (\I-\G_\a)^{-1}$ and the definition of~$\mb{b}^k$ in Proposition~\ref{LTI_ncvx}.
\end{proof}


\begin{lemma}\label{I-G}
If~$0<\alpha\leq\min\left\{\frac{(1-\lambda^2)^2}{48\lambda},\frac{\sqrt{n}}{\sqrt{8m}}\right\}\frac{1}{L}$, then
\begin{align*}
(\I_3 - \G_\a)^{-1} 
\leq&
\begin{bmatrix}
\star & \frac{776\lambda^2m\alpha^2L^2}{(1-\lambda^2)^3} & \frac{16\lambda^2\alpha^2L^2}{(1-\lambda^2)^3} \\
\star &  8m & \frac{114\lambda^2\alpha^2L^2}{(1-\lambda^2)^3}\\
\star & \star & \star
\end{bmatrix}, \nonumber\\
(\I_3-\G_\a)^{-1}\mb{b} \leq&
\begin{bmatrix}
\big(3104m^2+\frac{256\lambda^2}{1-\lambda^2}\big)\frac{\lambda^2\a^4L^2}{(1-\lambda^2)^3}\\
33m^2\alpha^2 \\
\star
\end{bmatrix},
\end{align*}
where the~$\star$ entries are not needed for further derivations.
\end{lemma}
\begin{proof}

In the following, for a matrix~$\mb{X}$, we denote~$\mb{X}^*$ as its adjugate and~$[\mb{X}]_{i,j}$ as its~$(i,j)$-th entry.
We first note that 
if~$0<\alpha\leq\tfrac{(1-\lambda^2)^2}{48\lambda L}$,
$\det\left(\I_3 - \G_\a\right) 
\geq \frac{(1-\lambda^2)^2}{32m}$.
We next derive upper bounds for entries of~$(\I_3-\G_\a)^*$:
\begin{align*}
&[(\I-\G_\a)^*]_{1,2} = \frac{97\lambda^2\alpha^2L^2}{4\left(1-\lambda^2\right)},
[(\I-\G_\a)^*]_{1,3} = \frac{\lambda^2\alpha^2L^2}{2m(1-\lambda^2)}, \\
&[(\I-\G_\a)^*]_{2,2} \leq \frac{(1-\lambda^2)^2}{4}, 
~[(\I-\G_\a)^*]_{2,3} = \frac{9\lambda^2\alpha^2L^2}{2m(1-\lambda^2)}.
\end{align*}
The upper bound on~$(\I_3-\G_\a)^{-1}$ then follows by using the above relations. Finally, we have:
\begin{align*}
(\I_3-\G_\a)^{-1}\mb{b} 
\leq
\begin{bmatrix}
\frac{3104\lambda^2m^2\alpha^4L^2}{(1-\lambda^2)^3}+\frac{256\lambda^4\alpha^4L^2}{(1-\lambda^2)^4}\\
32m^2\alpha^2 + \frac{2304\lambda^4\alpha^4L^2}{(1-\lambda^2)^4} \\
\star
\end{bmatrix}.
\end{align*}
If~$0<\alpha\leq\frac{(1-\lambda^2)^2}{48\lambda L}$, then
$32m^2\alpha^2 + \frac{2304\lambda^4\alpha^4L^2}{(1-\lambda^2)^4}\leq33m^2\a^2$ and the bound on~$(\I_3-\G_\a)^{-1}\mb{b}$  follows.
\end{proof}

We now bound two important quantities as follows.

\begin{lemma}\label{incremental_sums}
If~$0<\alpha\leq\min\Big\{\frac{(1-\lambda^2)^2}{48\lambda},\frac{\sqrt{n}}{\sqrt{8m}}\Big\}\frac{1}{L}$, then we have:~$\forall K\geq1$,
\begin{align}
&\sum_{k=0}^{K}\mbb{E}\left[\frac{1}{n}\|\mb{x}^k - \J\mb{x}^k\|^2\right]
\leq\frac{16\lambda^4\alpha^2}{(1-\lambda^2)^3}\frac{\|\nf(\x^0)\|^2}{n} \nonumber\\
&+ \left(97m^2+\dfrac{8\lambda^2}{1-\lambda^2}\right)\dfrac{32\lambda^2\a^4L^2}{(1-\lambda^2)^3}\sum_{k=0}^{K-1}\mbb{E}\big[\|\ol{\nabla \mb{f}}(\mb{x}^k)\|^2\big], \label{consensus_sum}
\end{align}
and,~$\forall K\geq1$,
\begin{align}
\sum_{k=0}^{K}\mbb{E}\left[t^k\right] \leq&~\dfrac{114\lambda^4\alpha^2}{(1-\lambda^2)^3}\frac{\|\nf(\x^0)\|^2}{n}
\nonumber\\
&+33m^2\alpha^2\sum_{k=0}^{K-1}\mbb{E}\big[\|\ol{\nabla \mb{f}}(\mb{x}^k)\|^2\big]. \label{tk_sum}
\end{align}
\end{lemma}
\begin{proof}
{\color{black}By~\eqref{gtsaga_y}, we have
$\|\mb{y}^1-\J\mb{y}^1\|^2 
=\|(\W-\J)(\mb{y}^0 + \g^{0} - \g^{-1})\|^2   
\leq\lambda^2\|\nf(\x^0)\|^2.
$
The proof then follows by applying this inequality and Lemma~\ref{I-G} to Lemma~\ref{incremental_sums_vec}.}
\end{proof}
Now, we are ready to prove Theorem~\ref{main_ncvx}.
\begin{P1}
We sum up the inequality in Lemma~\ref{descent} over~$k$ to obtain: if~$0<\alpha\leq\frac{1}{2L}$, then~$\forall K\geq1$,
\begin{align}\label{descent_ncvx_sum_0}
\mathbb{E}\big[F(\ol{\mb x}^{K})\big]  
\leq&~F(\ol{\mb x}^{0}) -\frac{\alpha}{2}\sum_{k=0}^{K-1}\mathbb{E}\big[\|\nabla F(\ol{\mb x}^{k})\|^2\big] \nonumber\\ &-\frac{\alpha}{4}\sum_{k=0}^{K-1}\mathbb{E}\big[\|\ol{\nabla\mb{f}}(\mb{x}^{k})\|^2\big] + \frac{\alpha^2L^3}{n}\sum_{k=0}^{K-1}\mathbb{E}\big[t^k\big]\nonumber\\
&+ \alpha L^2 \sum_{k=0}^{K-1}\mathbb{E}\Big[\frac{1}{n}\|\mb{x}^k - \J\mb{x}^k\|^2\Big] .
\end{align}
By the~$L$-smoothness of~$F$, we have:
$\frac{1}{2n}\sum_{i=1}^n\|\nabla F(\x_i^k)\|^2    
\leq \|\nabla F(\ol{\x}^k)\|^2 + \frac{L^2}{n}\|\x^k-\J\x^k\|^2$,~$\forall k\geq0$.
Using this inequality in~\eqref{descent_ncvx_sum_0}, we obtain: if~$0<\alpha\leq\frac{1}{2L}$, then~$\forall K\geq1$,
\begin{align}\label{descent_ncvx_sum_1}
\mathbb{E}\big[F(\ol{\mb x}^{K})\big] 
\leq&~F(\ol{\mb x}^{0})
-\frac{\alpha}{4n}\sum_{i=1}^n\sum_{k=0}^{K-1}\mathbb{E}\big[\|\nabla F(\x_i^{k})\|^2\big] \nonumber\\
&-\frac{\alpha}{4}\sum_{k=0}^{K-1}\mathbb{E}\big[\|\ol{\nabla\mb{f}}(\mb{x}^{k})\|^2\big]  + \frac{\alpha^2L^3}{n}\sum_{k=0}^{K-1}\mathbb{E}\big[t^k\big]\nonumber\\
&+ \frac{3\alpha L^2}{2} \sum_{k=0}^{K-1}\mathbb{E}\Big[\frac{1}{n}\|\mb{x}^k - \J\mb{x}^k\|^2\Big].
\end{align}
Applying~\eqref{tk_sum} to~\eqref{descent_ncvx_sum_1}, we obtain the following inequality: if~$0<\alpha\leq\min\Big\{\frac{(1-\lambda^2)^2}{48\lambda},\frac{\sqrt{n}}{\sqrt{8m}},\frac{1}{2}\Big\}\frac{1}{L}$, then $\forall K\geq1$,
\begin{align}\label{descent_ncvx_sum_2}
\mathbb{E}\big[F&(\ol{\mb x}^{K})\big]  
\leq F(\ol{\mb x}^{0})
-\frac{\alpha}{4n}\sum_{i=1}^n\sum_{k=0}^{K-1}\mathbb{E}\big[\|\nabla F(\x_i^{k})\|^2\big] \nonumber\\
&-\frac{\alpha}{8}\sum_{k=0}^{K-1}\mathbb{E}\big[\|\ol{\nabla\mb{f}}(\mb{x}^{k})\|^2\big] + \dfrac{114\lambda^4\alpha^4L^3}{n(1-\lambda^2)^3}\frac{\|\nf(\x^0)\|^2}{n}  \nonumber\\
& + \frac{3\alpha L^2}{2} \sum_{k=0}^{K-1}\mathbb{E}\Big[\frac{1}{n}\|\mb{x}^k - \J\mb{x}^k\|^2\Big] \nonumber\\
&-\frac{\alpha}{8}\left(1 - \frac{264m^2\alpha^3L^3}{n}\right)\sum_{k=0}^{K-1}\mbb{E}\big[\|\ol{\nabla \mb{f}}(\mb{x}^k)\|^2\big].
\end{align}
If~$0<\a\leq\frac{2n^{1/3}}{13m^{2/3}L}$, ${1 - \frac{264m^2\alpha^3L^3}{n}\geq0}$ and thus the last term in~\eqref{descent_ncvx_sum_2} may be dropped. {\color{black}We then use~\eqref{consensus_sum} in~\eqref{descent_ncvx_sum_2} to obtain: if $0<\alpha\leq\min\!\Big\{\!\frac{(1-\lambda^2)^2}{48\lambda},\frac{2n^{1/3}}{13m^{2/3}},\frac{1}{2}\!\Big\}\frac{1}{L}$, then~$\forall K\geq1$, 
\begin{align}\label{descent_ncvx_sum_3}
\mathbb{E}\big[F(\ol{\mb x}^{K})\big] 
\leq&~F(\ol{\mb x}^{0})
-\frac{\alpha}{4n}\sum_{i=1}^n\sum_{k=0}^{K-1}\mathbb{E}\big[\|\nabla F(\x_i^{k})\|^2\big] \nonumber\\
-&\frac{\alpha L^2}{4} \sum_{k=0}^{K-1}\mathbb{E}\left[\frac{1}{n}\|\mb{x}^k - \J\mb{x}^k\|^2\right] \nonumber\\
+&\left(\frac{114\a L}{28n}+1\right)\dfrac{28\lambda^4\alpha^3L^2}{(1-\lambda^2)^3}\frac{\|\nf(\x^0)\|^2}{n}
  \nonumber\\
-&\frac{\a}{8}\left(1 - \max\left\{97m^2,\dfrac{8\lambda^2}{1-\lambda^2}\right\}\dfrac{896\lambda^2\a^4L^4}{(1-\lambda^2)^3}\right) \nonumber\\
&\qquad\times\sum_{k=0}^{K-1}\mbb{E}\left[\|\ol{\nabla \mb{f}}(\mb{x}^k)\|^2\right].
\end{align}
We note that if~$0<\a\leq\min\left\{\frac{(1-\lambda^2)^{3/4}}{18\lambda^{1/2}m^{1/2}},\frac{1-\lambda^2}{12\lambda}\right\}\frac{1}{L}$, then $\max\big\{97m^2,\frac{8\lambda^2}{1-\lambda^2}\big\}\frac{896\lambda^2\a^4L^4}{(1-\lambda^2)^3}\leq1$ and the last term in~\eqref{descent_ncvx_sum_3} may be dropped. Therefore, if~$0<\alpha\leq\ol{\a}_1$ for~$\ol{\alpha}_1$ defined in Theorem~\ref{main_ncvx}, we obtain from~\eqref{descent_ncvx_sum_3} that~$\forall K\geq1$,
\begin{align}\label{descent_ncvx_sum_4}
&\mathbb{E}\big[F(\ol{\mb x}^{K})\big] 
\leq F(\ol{\mb x}^{0})
-\frac{\alpha}{4n}\sum_{i=1}^n\sum_{k=0}^{K-1}\mathbb{E}\big[\|\nabla F(\x_i^{k})\|^2\big] \nonumber\\ 
&-\frac{\alpha L^2}{4} \sum_{k=0}^{K-1}\mathbb{E}\Big[\frac{1}{n}\|\mb{x}^k - \J\mb{x}^k\|^2\Big]
+\dfrac{112\lambda^4\alpha^3L^2}{(1-\lambda^2)^3}\frac{\|\nf(\x^0)\|^2}{n}.
\end{align}
Since~$F$ is bounded below by~$F^*$,~\eqref{descent_ncvx_sum_4} leads to,~$\forall K\geq1$,    
\begin{align}\label{DS_final}
&\sum_{k=0}^{K-1}\frac{1}{n}\sum_{i=1}^n\mathbb{E}\Big[\|\nabla F(\x_i^{k})\|^2
+ L^2\|\mb{x}_i^k - \ol{\mb{x}}^k\|^2\Big]
\nonumber\\
\leq&~\frac{4(F(\ol{\mb x}^{0}) - F^*)}{\a}
+\dfrac{448\lambda^4\alpha^2L^2}{(1-\lambda^2)^3}\frac{\|\nf(\x^0)\|^2}{n}.
\end{align}
Since the RHS of~\eqref{DS_final} is finite and independent of~$K$, we let~$K\rightarrow\infty$ in~\eqref{DS_final} to obtain:
\begin{align}\label{DS_ms}
&\sum_{k=0}^{\infty}\sum_{i=1}^n\mathbb{E}\Big[\|\nabla F(\x_i^{k})\|^2
+ \|\mb{x}_i^k - \ol{\mb{x}}^k\|^2\Big]
<\infty,
\end{align}
which shows that all nodes in \textbf{\texttt{GT-SAGA}} asymptotically agree on a stationary point of~$F$ in the mean-squared sense. Moreover, since the series on the LHS of~\eqref{DS_ms} is nonnegative, we may exchange the order of the series and expectation to obtain~\cite{probability_with_martinagles}: $\mathbb{E}\left[\sum_{k=0}^{\infty}\sum_{i=1}^n(\|\nabla F(\x_i^{k})\|^2
+ \|\mb{x}_i^k - \ol{\mb{x}}^k\|^2)\right]
<\infty$, which implies that 
\begin{align}\label{DS_as}
\mathbb{P}\left(\sum_{k=0}^{\infty}\sum_{i=1}^n\Big(\|\nabla F(\x_i^{k})\|^2
+ \|\mb{x}_i^k - \ol{\mb{x}}^k\|^2\Big)<\infty\right)=1,    
\end{align}
i.e., all nodes in \textbf{\texttt{GT-SAGA}} asymptotically agree on a stationary point of~$F$ in the almost sure sense. Finally,
towards the iteration complexity of \textbf{\texttt{GT-SAGA}}, we set~$\alpha = \ol{\alpha}_1$ in~\eqref{DS_final} and divide the resulting inequality by~$K$ to obtain:~$\forall K\geq1$,
\begin{align}\label{DS_rate}
&\frac{1}{n}\sum_{i=1}^n\frac{1}{K}\sum_{k=0}^{K-1}\mathbb{E}\big[\|\nabla F(\x_i^{k})\|^2\big]
\nonumber\\
\leq&~\frac{4(F(\ol{\mb x}^{0}) - F^*)}{\ol{\a}_1 K}
+\dfrac{448\lambda^4\ol{\a}_1^2L^2}{(1-\lambda^2)^3 K}\frac{\|\nf(\x^0)\|^2}{n}.
\end{align}
Based on~\eqref{DS_rate}, the iteration complexity of \textbf{\texttt{GT-SAGA}} then follows by recalling the definition of~$\ol{\a}_1$ in Theorem~\ref{main_ncvx} and that~$\frac{448\lambda^4\ol{\a}_1^2L^2}{(1-\lambda^2)^3}\leq\frac{\lambda^2(1-\lambda^2)}{4}$ since~$0<\ol{\a}_1\leq\frac{(1-\lambda^2)^2}{48\lambda L}$.}
\end{P1}

\subsection{Proof of Theorem~\ref{main_PL}} \label{PL_sec}
In this subsection, we prove the linear rate of \textbf{\texttt{GT-SAGA}} when the global function~$F$ additionally satisfies the PL condition. In particular, we use the PL condition and Lemma~\ref{basic}(\ref{upper_L}) to refine the descent inequality in Lemma~\ref{descent} and the previously obtained LTI system in Proposition~\ref{LTI_ncvx}. 

\begin{lemma}\label{descent_PL}
If~$0<\alpha\leq\frac{1}{2L}$, then~$\forall k\geq0$,
\begin{align*}
&\mathbb{E}\big[F(\ol{\mb x}^{k+1}) - F^*|\F^k\big] \nonumber\\
\leq&~(1-\mu\a)(F(\ol{\mb{x}}^k)-F^*)
+ \frac{\alpha L^2}{n}\|\mb{x}^k - \J\mb{x}^k\|^2
+ \frac{\alpha^2L^3}{n}t^k.
\end{align*}
\end{lemma}
\begin{proof}
Apply the PL condition to Lemma~\ref{descent} and then subtract~$F^*$ from the resulting inequality.
\end{proof}

Next, we refine Corollary~\ref{aux_ctr} as follows.
\begin{lemma}\label{aux_PL}
If~$0<\alpha\leq\frac{\sqrt{n}}{\sqrt{8m}L}$, then~$\forall k\geq0$,
\begin{align*}
\mbb{E}\big[t^{k+1}|\mc{F}^k\big]
\leq&~
\Big(1-\frac{1}{4m}\Big)t^k
+16m\alpha^2L\left(F(\ol{\x}^k) - F^*\right) \nonumber\\
&+\Big(8m\alpha^2L^2+\frac{9}{4m}\Big)
\frac{1}{n}\|\mb{x}^{k}-\J\mb{x}^{k}\|^2.
\end{align*}
\end{lemma}
\begin{proof}
By Lemma~\ref{basic}(\ref{Lbound}) and~\ref{basic}(\ref{upper_L}), we have:~$\forall k\geq0$,
\begin{align}\label{PL_refine_0}
\|\ol{\nf}(\x^{k})\|^2 
\leq&~2\|\nabla F(\ol{\x}^k)\|^2 + 2\|\nabla F(\ol{\x}^k) - \ol{\nf}(\x^{k})\|^2  
\nonumber\\
\leq&~4L\left(F(\ol{\x}^k) - F^*\right) + \frac{2L^2}{n}\|\x^k -\J\x^k\|^2.     
\end{align}
The proof follows by applying~\eqref{PL_refine_0} to Corollary~\ref{aux_ctr}.
\end{proof}

We finally refine Lemma~\ref{GT_final} as follows.
\begin{lemma}\label{GT_PL}
If~$0<\a\leq\min\left\{\frac{1-\lambda^2}{16\lambda},\frac{\sqrt{n}}{\sqrt{8m}}\right\}\frac{1}{L}$, then~$\forall k\geq0$,
\begin{align*}
&\E\big[\|\mb{y}^{k+2}-\J\mb{y}^{k+2}\|^2\big] \nonumber\\
\leq&~\frac{1+\lambda^2}{2}\E\big[\|\mb{y}^{k+1}-\J\mb{y}^{k+1}\|^2\big]
+ \frac{31L^2}{1-\lambda^2}\E\big[\|\x^{k}-\J\x^{k}\|^2\big]  \nonumber\\
&+ \frac{97L^2n}{8}\E\big[t^k\big] + \frac{64\lambda^2\a^2L^3n}{1-\lambda^2}\E\big[F(\ol{\x}^k) - F^*\big].  \end{align*}
\end{lemma}
\begin{proof}
Applying~\eqref{PL_refine_0} to Lemma~\ref{GT_final}, we have: if~$0<\a\leq\min\left\{\frac{1-\lambda^2}{16\lambda},\frac{\sqrt{n}}{\sqrt{8m}}\right\}\frac{1}{L}$, then~$\forall k\geq0$,
\begin{align*}
\E\big[\|\mb{y}^{k+2}&-\J\mb{y}^{k+2}\|^2\big] 
\leq\frac{1+\lambda^2}{2}\E\big[\|\mb{y}^{k+1}-\J\mb{y}^{k+1}\|^2\big]
\nonumber\\
&+ \left(30.5+ 32\lambda^2\a^2L^2\right)\frac{L^2}{1-\lambda^2}\E\big[\|\x^{k}-\J\x^{k}\|^2\big] \nonumber\\
&+ \frac{97L^2n}{8}\E\big[t^k\big] 
+ \frac{64\lambda^2\a^2L^3n}{1-\lambda^2}\E\big[F(\ol{\x}^k) - F^*\big]. 
\end{align*}
We conclude by~$30.5+32\lambda^2\a^2L^2\leq 31$ if~$0<\a\leq\frac{1-\lambda^2}{16\lambda L}$.
\end{proof}

Now, we write~\eqref{consensus1}, Lemma~\ref{descent_PL},~\ref{aux_PL} and~\ref{GT_PL} in a LTI system.
\begin{prop}
If~$0<\a\leq\min\left\{\frac{1-\lambda^2}{16\lambda},\frac{\sqrt{n}}{\sqrt{8m}},\frac{1}{2}\right\}\frac{1}{L}$, then
$$
\mb{v}^{k+1} \leq \H_{\a}\mb{v}^{k}, \qquad\forall k\geq0, 
$$
where~$\mb{v}^{k}\in\mbb{R}^4$ and $\H_{\a}\in\mbb{R}^{4\times4}$ are given by
\begin{align*}
\mb{v}^k :=& 
\begin{bmatrix}
\mbb{E}\left[\frac{1}{n}\|\mb{x}^{k}-\J\mb{x}^{k}\|^2\right]\\
\frac{1}{L}\mathbb{E}\left[F(\ol{\x}^k)-F^*\right]\\
\mathbb{E}\left[t^{k}\right]\\
\mbb{E}\left[\frac{1}{nL^2}\|\mb{y}^{k+1}-\J\mb{y}^{k+1}\|^2\right]
\end{bmatrix},
\\
\H_\a :=&
\begin{bmatrix}
\frac{1+\lambda^2}{2} & 0 & 0 & \frac{2\lambda^2\alpha^2L^2}{1-\lambda^2} \\
\a L&  1-\mu\a & \frac{\a^2L^2}{n} & 0\\
8m\a^2L^2+\frac{9}{4m} & 16m\a^2L^2 & 1-\frac{1}{4m} & 0 \\
\frac{31}{1-\lambda^2} & \frac{64\lambda^2\a^2L^2}{1-\lambda^2} & \frac{97}{8} & \frac{1+\lambda^2}{2}
\end{bmatrix}.
\end{align*}
\end{prop}

We are ready to prove Theorem~\ref{main_PL}, i.e., to establish an upper bound on~$\rho(\H_\a)$ that characterizes the explicit linear rate of \textbf{\texttt{GT-SAGA}} under the PL condition.
\begin{P2}
In light of Lemma~\ref{rho_bound}, we solve for the range of~$\a$ under which there exists a positive vector~$\mb{s}_\a = [s_1,s_2,s_3,s_4]^\top$ s.t.~$\H_\a\mb{s}_\a \leq (1-\frac{\mu\a}{2})\mb{s}_\a$, i.e.,
\begin{align}
&\frac{2\lambda^2\a^2L^2}{1-\lambda^2}s_4 \leq\Big(\frac{1-\lambda^2}{2}-\frac{\mu\a}{2}\Big)s_1,        \label{s1}\\
&\a Ls_1 + \frac{\a^2L^2}{n}s_3 \leq\frac{\mu\a}{2}s_2, 
\label{s2}\\
&\Big(8m\a^2L^2+\frac{9}{4m}\Big)s_1 + 16m\a^2L^2s_2
\leq\frac{1-2m\mu\a}{4m}s_3, \label{s3} \\
&\dfrac{31}{1-\lambda^2}s_1 + \frac{64\lambda^2\a^2L^2}{1-\lambda^2}s_2 + \frac{97}{8}s_3 \leq\frac{1-\lambda^2-\mu\a}{2}s_4. \label{s4} 
\end{align}
We first note that~\eqref{s2} is equivalent to~$\frac{\a L^2}{n}s_3 \leq\frac{\mu}{2}s_2 - Ls_1$, based on which we set the values of~$s_1,s_2,s_3$ as
\begin{align}\label{s123}
s_1 = 1/(4\kappa), \qquad s_2 = 1, \qquad s_3 = n/(4\a\kappa L),    
\end{align}
where~$\kappa = L/\mu$. Next, we write~\eqref{s3} equivalently as
\begin{align}\label{s3_1}
8m\a^2L^2\left(s_1 + 2s_2\right) 
\leq\frac{1-2m\mu\a}{4m}s_3 - \frac{9}{4m}s_1.   
\end{align}
According to~\eqref{s3_1}, we enforce~$0<\a\leq\frac{1}{4m\mu}$,
i.e.,~$\frac{1-2m\mu\a}{4m}\geq\frac{1}{8m}$; therefore to make~\eqref{s3_1} hold, with the help of the values of~$s_1,s_2,s_3$ in~\eqref{s123}, it suffices to further choose~$\a$ such that
\begin{align}\label{s3_2}
18m\a^2L^2 
\leq&~\frac{1}{16m\kappa}\left(\frac{n}{2\a L} - 9\right).   
\end{align}
According to~\eqref{s3_2}, we enforce~$0<\a\leq\frac{n}{36L}$, i.e., $\frac{n}{2\a L} - 9\geq\frac{n}{4\a L}$, and therefore to make~\eqref{s3_2} hold, it suffices to further choose~$\a$ such that~$0< \a \leq \frac{n^{1/3}}{10.5m^{2/3}\kappa^{1/3}L}$. Next, according to~\eqref{s4} we further enforce~$0<\a \leq \frac{1-\lambda^2}{2\mu}$, i.e.,~$\frac{1-\lambda^2-\mu\a}{2} \geq \frac{1-\lambda^2}{4}$ and therefore to make~\eqref{s4} hold we set~$s_4$ as
\begin{align}\label{s_4}
s_4 
=&~\dfrac{124}{(1-\lambda^2)^2}s_1 + \frac{256\lambda^2\a^2L^2}{(1-\lambda^2)^2}s_2 + \frac{97}{2(1-\lambda^2)}s_3.    
\end{align}
Finally, since~$0<\a\leq\frac{1-\lambda^2}{2\mu}$, to make~\eqref{s1} hold, it suffices to further choose~$\a$ such that~$
\frac{8\lambda^2\a^2L^2}{(1-\lambda^2)^2}\frac{s_4}{s_1} \leq 1$, which, 
using the values of~$s_1,s_4$, becomes
\begin{align}\label{s1_2}
\frac{992\lambda^2\a^2L^2}{(1-\lambda^2)^4} 
+ \frac{8192\kappa\lambda^4\a^4L^4}{(1-\lambda^2)^4} 
+ \frac{388\lambda^2n\a L}{(1-\lambda^2)^3}
\leq 1.
\end{align}
If $0<\a\leq\min\left\{\frac{(1-\lambda^2)^2}{55\lambda}, \frac{1-\lambda^2}{13\lambda\kappa^{1/4}},\frac{(1-\lambda^2)^3}{388\lambda^2n}\right\}\frac{1}{L}$, then the terms on the LHS of~\eqref{s1_2} are respectively less than~$\frac{1}{3}$ and thus~\eqref{s1_2} holds. 
Based on the above derivations and Lemma~\ref{rho_bound}, we have: if~$0<\alpha\leq\ol{\alpha}_2$ for~$\ol{\alpha}_2$ defined in Theorem~\ref{main_PL}, 
then~$\rho(\H_\a) \leq 1 - \frac{\mu\a}{2}$ which concludes the proof.
\end{P2}

\section{Conclusion} \label{clu}
In this paper, we analyze \textbf{\texttt{GT-SAGA}}, a decentralized randomized incremental gradient method that combines node-level variance reduction and network-level gradient tracking. For both general smooth non-convex problems and problems where the global function additionally satisfies the PL condition, we prove that \textbf{\texttt{GT-SAGA}} achieves fast convergence rate. We further identify practical regimes where \textbf{\texttt{GT-SAGA}} outperforms the existing approaches. We also present numerical simulations to verify the theoretical results in this paper. {\color{black}Future research includes generalization of \textbf{\texttt{GT-SAGA}} to the setting of time-varying directed networks~\cite{TV-AB} and of zeroth order gradient computation~\cite{zero_Anit,zero_Hu,GT_zero}. It is also of interest to incorporate weighted sampling techniques~\cite{Prox_SVRG} in \textbf{\texttt{GT-SAGA}} to improve the dependence of the smoothness parameters of the component functions on the convergence rate.} 

\bibliography{reference.bib}
\bibliographystyle{IEEEbib}

\begin{IEEEbiography}[{\includegraphics[width=1in,height=1.2in,clip,keepaspectratio]{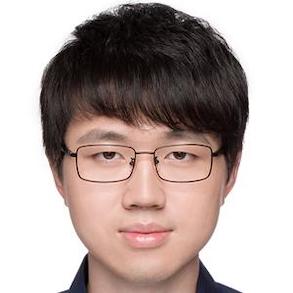}}]{Ran Xin} received his B.S. degree in Mathematics and Applied Mathematics from Xiamen University, China, in 2016, and M.S. degree in Electrical and Computer Engineering from Tufts University in 2018. Currently, he is a Ph.D. candidate in the Electrical and Computer Engineering Department at Carnegie Mellon University. His research interests include convex and non-convex optimization, stochastic approximation, and machine learning.
\end{IEEEbiography}

\begin{IEEEbiography}[{\includegraphics[width=1in,height=1.2in,clip,keepaspectratio]{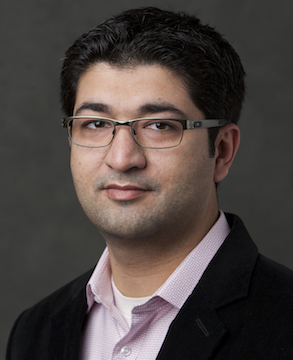}}]{Usman A. Khan} is an Associate Professor of Electrical and Computer Engineering (ECE) at Tufts University, Medford, MA, USA. His research interests include statistical signal processing, network science, and decentralized optimization over multi-agent systems. 
He received his B.S. degree in 2002 from University of Engineering and Technology, Pakistan, M.S. degree in 2004 from University of Wisconsin-Madison, USA, and Ph.D. degree in 2009 from Carnegie Mellon University, USA, all in ECE. 
He 
is currently an Associate Editor of the \textit{IEEE Control System Letters}, \textit{IEEE Transactions on Signal and Information Processing over Networks}, and \textit{IEEE Open Journal of Signal Processing}. 
\end{IEEEbiography}

\begin{IEEEbiography}[{\includegraphics[width=1in,height=1.2in,clip,keepaspectratio]{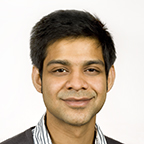}}]{Soummya Kar} received a B.Tech. in Electronics and Electrical Communication Engineering from the Indian Institute of Technology, Kharagpur, India, in May 2005 and a Ph.D. in electrical and computer engineering from Carnegie Mellon University, Pittsburgh, PA, in 2010. From June 2010 to May 2011 he was with the Electrical Engineering Department at Princeton University as a Postdoctoral Research Associate. He is currently a Professor of Electrical and Computer Engineering at Carnegie Mellon University. His research interests span several aspects of decision-making in large-scale networked dynamical systems with applications to problems in network science, cyber-physical systems and energy systems. 

\end{IEEEbiography}

\end{document}